\documentclass{llncs}                     

\usepackage{amsmath}
\usepackage{amsfonts}
\usepackage{amssymb}
\usepackage{array}
\usepackage{eqparbox}
\usepackage[caption=false,font=footnotesize]{subfig}
\usepackage{mathptmx}
\usepackage{stmaryrd}
\usepackage{tikz}
  \usetikzlibrary{arrows,automata,shapes}
  \usetikzlibrary{decorations.pathmorphing}
  \usetikzlibrary{positioning}
  \usetikzlibrary{decorations.markings}
  \usetikzlibrary{intersections}

\newtheorem{algorithm}[theorem]{Algorithm}

\newcommand{\eps}{\varepsilon}
\newcommand{\supC}{\mbox{$\sup {\rm C}$}}
\newcommand{\supCC}{\mbox{$\sup {\rm cC}$}}

\begin{document}
\title{Coordination Control of Discrete-Event Systems Revisited%
  \thanks{A preliminary version was presented at the 11th International Workshop on Discrete Event Systems (WODES 2012) held in Guadalajara, Mexico~\cite{wodes12}.}
}

\author{Jan~Komenda\,\inst{1} \and Tom{\' a}{\v s}~Masopust\,\inst{1} \and Jan~H.~van~Schuppen\,\inst{2,}\thanks{Most of this work was done when the author was with CWI, Amsterdam, The Netherlands.}}

\institute{
    Institute of Mathematics,
    Academy of Sciences of the Czech Republic\\
    {\v Z}i{\v z}kova 22,
    616~62~Brno, Czech Republic\\
    \email{komenda@ipm.cz, masopust@math.cas.cz}
  \and
    Van Schuppen Control Research\\
    Gouden Leeuw 143,
    1103 KB Amsterdam,
    The Netherlands\\
    \email{jan.h.van.schuppen@euronet.nl}
}

\maketitle
\begin{abstract}
  In this paper, we revise and further investigate the coordination control approach proposed for supervisory control of distributed discrete-event systems with synchronous communication based on the Ramadge-Wonham automata fra\-me\-work. The notions of conditional decomposability, conditional controllability, and conditional closedness ensuring the existence of a solution are carefully revised and simplified. The paper is generalized to non-prefix-closed languages, that is, supremal conditionally controllable sublanguages of not necessary prefix-closed languages are discussed. Non-prefix-closed languages introduce the blocking issue into coordination control, hence a procedure to compute a coordinator for nonblockingness is included. The optimization problem concerning the size of a coordinator is under investigation. We prove that to find the minimal extension of the coordinator event set for which a given specification language is conditionally decomposable is NP-hard. In other words, unless P=NP, it is not possible to find a polynomial algorithm to compute the minimal coordinator with respect to the number of events.
\end{abstract}
\begin{keywordname}
  Discrete-event systems, distributed systems with synchronous communication, supervisory control, coordination control, conditional decomposability.
\end{keywordname}

\section{Introduction}
  In this paper, we revise and further investigate the coordination control approach proposed for supervisory control of distributed discrete-event systems with synchronous communication based on the Ramadge-Wonham automata framework. A distributed discrete-event system with synchronous communication is modeled as a parallel composition of two or more subsystems, each of which has its own observation channel. The local control synthesis consists in synthesizing local nonblocking supervisors for each of the subsystems. It is well-known that such a purely decentralized (often referred to as modular) approach does not work in general. Recently, Komenda and Van Schuppen~\cite{KvS08} have proposed a coordination control architecture as a trade-off between the purely local control synthesis, which is not effective in general because the composition of local supervisors may violate the specification, and the global control synthesis, which is not always possible because of the complexity reasons since the composition of all subsystems can result in an exponential blow-up of states in the monolithic plant. The coordination control approach has been developed for prefix-closed languages in~\cite{automatica2011} and extended to systems with partial observations in~\cite{scl2011}. The case of non-prefix-closed languages has partially been discussed in~\cite{ifacwc2011}. Most of these approaches for prefix-closed languages have already been implemented in the software library libFAUDES~\cite{faudes}.
  
  In the last two decades several alternative approaches have been proposed for supervisory control of large discrete-event systems. Among the different control architectures are such as hierarchical control based on abstraction~\cite{KS,WW96,WZ91}, modular approaches~\cite{FLT,GM04,KvSGM08,pcl06}, decentralized control~\cite{RW92,YLL02} also with inferencing (conditional decisions)~\cite{KT07,YLL04} or with communicating supervisors~\cite{RR00}, and the so-called interface-based approach~\cite{LLW05}. Nowadays, these approaches are combined to achieve even better results, cf.~\cite{SB08,SB11}. Our coordination control approach can be seen as a combination of the horizontal and vertical modularity. The coordinator level corresponds to the abstraction (i.e., the higher level) of hierarchical control, while the local control synthesis is a generalization of the modular control synthesis. Moreover, coordination control is closely related to decentralized control with communication, because local supervisors communicate indirectly via a coordinator, cf.~\cite{Barrett}.
  
  In this paper, the notions of conditional decomposability, conditional controllability, and conditional closedness, which are the central notions to characterize the solvability of the coordination control problem, are carefully revised and simplified. The paper is generalized to non-prefix-closed languages, hence supremal conditionally controllable sublanguages of not necessary prefix-closed languages are discussed. This generality, however, introduces the problem of nonblockingness into the coordination control approach, therefore a part with a procedure to compute a coordinator for nonblockingness is included in the paper. The optimization problem concerning the size of a coordinator is nowadays the main problem under investigation. The construction of a coordinator described in this paper depends mainly on a set of events, including the set of all shared events. We prove that to construct the coordinator so that its event set is minimal with respect to the number of events or, in other words, to find the minimal extension of the coordinator event set for which a given specification language is conditionally decomposable, is NP-hard.
  
  The main contributions and the organization of the paper are as follows. Section~\ref{sec:preliminaries} recalls the basics of supervisory control theory and revises the fundamental concepts. 
  Section~\ref{sec:nphard} gives the computational complexity analysis of the minimal extension problem for conditional decomposability and proves that it is NP-hard to find the minimal extension with respect to set inclusion (Corollary~\ref{NPhard}). 
  Section~\ref{sec:controlsynthesis} formulates the problem of coordination supervisory control. The notion of conditional controllability (Definition~\ref{def:conditionalcontrollability}) is revised and simplified, however still equivalent to the previous definition in, e.g., \cite{automatica2011}.
  Section~\ref{sec:procedure} provides results concerning non-prefix-closed languages. Theorem~\ref{thm2} shows that in a special case the parallel composition of local supervisors results in the supremal conditionally controllable languages. However, the problem how to compute the supremal conditionally controllable sublanguage in general is open.
  Section~\ref{sec:nonblocking} discusses the construction of a coordinator for nonblockingness (Theorem~\ref{thm22}) and presents an algorithm. Section~\ref{sec:revisited} revises the prefix-closed case, where a less restrictive condition, LCC, is used instead of OCC. The possibility to use LCC instead of OCC has already been mentioned in~\cite{automatica2011} without proofs, therefore the proofs are provided here.
  Finally, Section~\ref{sec:conclusion} concludes the paper.

\section{Preliminaries and definitions}\label{sec:preliminaries}
  We assume that the reader is familiar with the basic notions and concepts of supervisory control of discrete-event systems modeled by deterministic finite automata with partial transition functions. For unexplained notions, the reader is referred to the monograph~\cite{CL08}.
  
  Let $\Sigma$ be a finite nonempty set whose elements are called {\em events}, and let $\Sigma^*$ denote the set of all finite words (finite sequences of events) over $\Sigma$; the {\em empty word\/} is denoted by $\eps$. Let $|\Sigma|$ denote the cardinality of $\Sigma$.
  
  A {\em generator\/} is a quintuple $G=(Q,\Sigma, f, q_0, Q_m)$, where $Q$ is a finite nonempty set of {\em states}, $\Sigma$ is a finite set of events (an {\em event set\/}), $f: Q \times \Sigma \to Q$ is a {\em partial transition function}, $q_0 \in Q$ is the {\em initial state}, and $Q_m\subseteq Q$ is a set of {\em marked states}. In the usual way, the transition function $f$ can be extended to the domain $Q \times \Sigma^*$ by induction. The behavior of generator $G$ is described in terms of languages. The language {\em generated\/} by $G$ is the set $L(G) = \{s\in \Sigma^* \mid f(q_0,s)\in Q\}$, and the language {\em marked\/} by $G$ is the set $L_m(G) = \{s\in \Sigma^* \mid f(q_0,s)\in Q_m\}$. Obviously, $L_m(G)\subseteq L(G)$.

  A {\em (regular) language\/} $L$ over an event set $\Sigma$ is a set $L\subseteq \Sigma^*$ such that there exists a generator $G$ with $L_m(G)=L$. The prefix closure of a language $L$ over $\Sigma$ is the set $\overline{L}=\{w\in \Sigma^* \mid \text{there exists } u \in\Sigma^* \text{ such that } wu\in L\}$ of all prefixes of words of the language $L$. A language $L$ is {\em prefix-closed\/} if $L=\overline{L}$.

  A {\em controlled generator\/} over an event set $\Sigma$ is a triple $(G,\Sigma_c,\Gamma)$, where $G$ is a generator over $\Sigma$, $\Sigma_c\subseteq\Sigma$ is a set of {\em controllable events}, $\Sigma_{u} = \Sigma \setminus \Sigma_c$ is the set of {\em uncontrollable events}, and $\Gamma = \{\gamma \subseteq \Sigma \mid \Sigma_{u} \subseteq \gamma\}$ is the {\em set of control patterns}. A {\em supervisor\/} for the controlled generator $(G,\Sigma_c,\Gamma)$ is a map $S:L(G) \to \Gamma$. The {\em closed-loop system\/} associated with the controlled generator $(G,\Sigma_c,\Gamma)$ and the supervisor $S$ is defined as the minimal language $L(S/G)$ such that the empty word $\eps$ belongs to $L(S/G)$, and for any word $s$ in $L(S/G)$ such that $sa$ is in $L(G)$ and $a$ in $S(s)$, the word $sa$ also belongs to $L(S/G)$. We define the marked language of the closed-loop system as the intersection $L_m(S/G) = L(S/G)\cap L_m(G)$. The intuition is that the supervisor disables some of the transitions of the generator $G$, but it can never disable any transition under an uncontrollable event. If the closed-loop system is nonblocking, which means that $\overline{L_m(S/G)}=L(S/G)$, then the supervisor $S$ is called {\em nonblocking}.

  Given a specification language $K$ and a plant (generator) $G$, the control objective of supervisory control is to find a nonblocking supervisor $S$ such that $L_m(S/G)=K$. For the monolithic case, such a supervisor exists if and only if the specification $K$ is both {\em controllable\/} with respect to the plant language $L(G)$ and uncontrollable event set $\Sigma_u$, that is the inclusion $\overline{K}\Sigma_u\cap L\subseteq \overline{K}$ is satisfied, and {\em $L_m(G)$-closed}, that is the equality $K = \overline{K}\cap L_m(G)$ is satisfied. For uncontrollable specifications, controllable sublanguages of the specification are considered instead. The notation $\supC(K,L(G),\Sigma_u)$ denotes the supremal controllable sublanguage of the specification $K$ with respect to the plant language $L(G)$ and uncontrollable event set $\Sigma_u$, which always exists and is equal to the union of all controllable sublanguages of the specification $K$, see~\cite{Won04}.

  A {\em (natural) projection\/} $P: \Sigma^* \to \Sigma_0^*$, where $\Sigma_0$ is a subset of $\Sigma$, is a homomorphism defined so that $P(a)=\eps$ for $a$ in $\Sigma\setminus \Sigma_0$, and $P(a)=a$ for $a$ in $\Sigma_0$. The projection of a word is thus uniquely determined by projections of its letters. The {\em inverse image\/} of $P$ is denoted by $P^{-1}:\Sigma_0^* \to 2^{\Sigma^*}$. For three event sets $\Sigma_i$, $\Sigma_j$, $\Sigma_\ell$, subsets of $\Sigma$, we use the notation $P^{i+j}_{\ell}$ to denote the projection from $(\Sigma_i\cup \Sigma_j)^*$ to $\Sigma_\ell^*$. If $\Sigma_i\cup \Sigma_j=\Sigma$, we simplify the notation to $P_\ell$. Similarly, the notation $P_{i+k}$ stands for the projection from $\Sigma^*$ to $(\Sigma_i\cup\Sigma_k)^*$. The projection of a generator $G$, denoted by $P(G)$, is a generator whose behavior satisfies $L(P(G))=P(L(G))$ and $L_m(P(G))=P(L_m(G))$.
  
  The synchronous product of languages $L_1$ over $\Sigma_1$ and $L_2$ over $\Sigma_2$ is defined as the language $L_1\parallel L_2=P_1^{-1}(L_1) \cap P_2^{-1}(L_2)$, where $P_i: (\Sigma_1\cup \Sigma_2)^*\to \Sigma_i^*$ is a projection, for $i=1,2$. A similar definition for generators can be found in~\cite{CL08}. The relation between the language definition and the generator definition is specified by the following equations. For generators $G_1$ and $G_2$, $L(G_1 \| G_2) = L(G_1) \parallel L(G_2)$ and $L_m(G_1 \| G_2)= L_m(G_1) \parallel L_m(G_2)$. In the automata framework, where a supervisor $S$ has a finite representation as a generator, the closed-loop system is a synchronous product of the supervisor and the plant. Thus, we can write the closed-loop system as $L(S/G)=L(S) \parallel L(G)$.

  For a generator $G$ over an event set $\Sigma$, let $\Sigma_r(G)=\{a\in \Sigma \mid$ there are words $u,v\in \Sigma^* \text{ such that } uav\in L(G)\}$ denote the set of all events appearing in words of the language $L(G)$. Generators $G_1$ and $G_2$ are {\em conditionally independent\/} with respect to a generator $G_k$ if all events shared by the subsystems appear in the generator $G_k$, that is, if the inclusion $\Sigma_r(G_1) \cap \Sigma_r(G_2) \subseteq \Sigma_r(G_k)$ is satisfied. In other words, there is no simultaneous move in both generators $G_1$ and $G_2$ without the generator $G_k$ being also involved. 

  Now, the notion of conditional decomposability is simplified compared to our previous work~\cite{automatica2011}, but still equivalent. 
  \begin{definition}
    A language $K$ is {\em conditionally decomposable\/} with respect to event sets $\Sigma_1$, $\Sigma_2$, $\Sigma_k$, where $\Sigma_1\cap\Sigma_2\subseteq\Sigma_k\subseteq \Sigma_1\cup\Sigma_2$, if 
    \[
      K = P_{1+k}(K)\parallel P_{2+k}(K)\,,
    \]
    where $P_{i+k}:(\Sigma_1\cup\Sigma_2)^*\to (\Sigma_i\cup\Sigma_k)^*$ is a projection, for $i=1,2$.
  \end{definition}

  Note that there always exists an extension of $\Sigma_k$ which satisfies this condition; $\Sigma_k = \Sigma_1 \cup \Sigma_2$ is a trivial example. Here the index $k$ is related to projection $P_k$ used later in the paper. There exists a polynomial algorithm to check this condition, and to extend the event set to satisfy the condition, see~\cite{SCL12}. However, the question which extension is the most appropriate requires further investigation. In Section~\ref{sec:nphard}, we show that to find the minimal extension is NP-hard.

  Languages $K$ and $L$ are {\em synchronously nonconflicting\/} if $\overline{K \parallel L} = \overline{K} \parallel \overline{L}$. 
  \begin{lemma}
    Let $K$ be a language. If the language $\overline{K}$ is conditionally decomposable, then the languages $P_{1+k}(K)$ and $P_{2+k}(K)$ are synchronously nonconflicting.
  \end{lemma}
  \begin{proof}
    Assume that the language $\overline{K}$ is conditionally decomposable. From a simple observation that $K \subseteq P_{i+k}^{-1}(P_{i+k}(K))$, for $i=1,2$, we immediately obtain that $K \subseteq P_{1+k}(K) \parallel P_{2+k}(K)$. As the prefix-closure is a monotone operation, 
    \[
      \overline{K} \subseteq \overline{P_{1+k}(K)\parallel P_{2+k}(K)} \subseteq \overline{P_{1+k}(K)} \parallel \overline{P_{2+k}(K)} = \overline{K}\,,
    \]
    which proves the lemma.
  \qed\end{proof}
  
  The following example shows that there exists, in general, no relation between the conditional decomposability of languages $K$ and $\overline{K}$.
  \begin{example}\label{ex1}
    Let $\Sigma_{1}=\{a_1,b_1,a,b\}$, $\Sigma_{2}=\{a_2,b_2,a,b\}$, and $\Sigma_k=\{a,b\}$ be event sets, and define the language $K=\{a_1a_2a,a_2a_1a,b_1b_2b,b_2b_1b\}$. Then 
      $P_{1+k}(K) = \{a_1a,b_1b\}$, 
      $P_{2+k}(K) = \{a_2a,b_2b\}$, and 
      $K = P_{1+k}(K)\parallel P_{2+k}(K)$. 
    Notice that whereas $a_1b_2$ is in $\overline{P_{1+k}(K)} \parallel \overline{P_{2+k}(K)}$, $a_1b_2$ is not in $\overline{K}$, which means that the language $\overline{K}$ is not conditionally decomposable.
    
    On the other hand, consider the language $L=\{\eps,ab,ba,abc,bac\}$ over the event set $\{a,b,c\}$ with $\Sigma_{1}=\{a,c\}$, $\Sigma_{2}=\{b,c\}$, $\Sigma_k=\{c\}$. Then 
    $
      \overline{L} = \overline{P_{1+k}(L)} \parallel \overline{P_{2+k}(L)} = P_{1+k}(L) \parallel P_{2+k}(L)\,,
    $
    and it is obvious that $L\neq \overline{L}$. 
    \hfill$\triangleleft$
  \end{example}

\section{Conditional decomposability minimal extension problem}\label{sec:nphard}
  We have defined conditional decomposability only for two event sets, but the definition can be extended to more event sets as follows. A language $K$ is {\em conditionally decomposable\/} with respect to event sets $(\Sigma_i)_{i=1}^{n}$, for some $n\ge 2$, and an event set $\Sigma_k$, where $\Sigma_k\subseteq \cup_{i=1}^{n} \Sigma_i$ contains all shared events, that is, it satisfies 
  \[
    \Sigma_s:=\bigcup_{i\neq j} (\Sigma_i\cap \Sigma_j)\subseteq\Sigma_k\,,
  \]
  if 
  \[
    K = \bigparallel _{i=1}^{n}P_{i+k}(K)\,.
  \]
  
  The conditional decomposability minimal extension problem is to find a minimal extension (with respect to set inclusion) of the event set $\Sigma_s$ of all shared events so that the language is conditionally decomposable with respect to given event sets and the extension of $\Sigma_s$. The optimization problem can be reformulated to a decision version as follows.
  \begin{problem}[CD MIN EXTENSION]\label{cd_dv}$ $\\
    INSTANCE: A language $K$ over an event set $\Sigma=\cup_{i=1}^{n} \Sigma_i$, where $n\ge 2$, and a positive integer $r\le |\Sigma|$.\\
    QUESTION: Is the language $K$ conditionally decomposable with respect to event sets $(\Sigma_i)_{i=1}^{n}$ and $\Sigma_s\cup \Sigma_r$, where $|\Sigma_r|\le r$?
  \end{problem}

  We now prove that the CD MIN EXTENSION problem is NP-complete. This then immediately implies that the optimization problem of finding the minimal extension of the event set $\Sigma_s$ is NP-hard. On the other hand, it is not hard to see that the optimization problem is in PSPACE. Indeed, we can check all subsets generated one by one using the polynomial algorithm described in~\cite{SCL12}.
  
  To prove NP-completeness, we reduce the MINIMUM SET COVER problem to the CD MIN EXTENSION problem; the MINIMUM SET COVER problem is NP-complete~\cite{GareyJ79}. 
  \begin{problem}[MINIMUM SET COVER]\label{msc}$ $\\
    INSTANCE: A collection $C$ of subsets of a finite set $S$, and a positive integer $t\le |C|$.\\
    QUESTION: Does the collection $C$ contain a cover for the set $S$ of cardinality $t$ or less, that is, a subset $C'$ with $|C'|\le t$ such that every element of the set $S$ belongs to at least one member of $C'$?
  \end{problem}

  \begin{theorem}
    The CD MIN EXTENSION problem is NP-complete.
  \end{theorem}
  \begin{proof}
    First, we show that CD MIN EXTENSION is in NP. To do this, a Turing machine guesses a set $\Sigma_r$ of cardinality at most $r$ and uses Algorithm~1 of~\cite{SCL12} to verify in polynomial time whether the given language is conditionally decomposable with respect to the given event sets.
    
    To prove the NP-hardness, consider an instance $(S,C)$ of the MINIMUM SET COVER problem as defined in Problem~\ref{msc} such that the union of all elements of the collection $C$ covers the set $S$ (otherwise it is trivial to solve the problem). Denote 
    \begin{eqnarray*}
      S=\{b_1,b_2,\ldots,b_n\} & \text{ and } & C=\{c_1,c_2,\ldots,c_m\}\,.
    \end{eqnarray*} 
    We now construct a language $K$ over the event set $S\cup \{a_i \mid i=1,2,\ldots,n\}\cup C\cup \{a\}$ as follows. For each $b_i$ in $S$, let $C_{b_i} = \{c_j \mid b_i\in c_{j}\}$ be the set of all elements of the collection $C$ containing the element $b_i$. Then, for $C_{b_i}=\{c_{i_1}, c_{i_2},\ldots, c_{i_{b_i}}\}$, where we assume without loss of generality that $i_1< i_2< \ldots < i_{b_i}$, add the two words $a_iab_i$ and $a_i c_{i_1} c_{i_2} \ldots c_{i_{b_i}} a$ to the language $K$. Then the language $K$ is
    \[
      K=\sum_{i=1}^{n} (a_i a b_i + a_i c_{i_1} c_{i_2} \ldots c_{i_{b_i}} a )\,.
    \]
    To demonstrate the construction, let $S=\{b_1,b_2,b_3,b_4,b_5\}$ and $C=\{c_1=\{b_1,b_2,b_3\},$ $c_2=\{b_2,b_4\},c_3=\{b_3,b_4\},c_4=\{b_4,b_5\}\}$. The generator for language $K$ is depicted in Fig.~\ref{fig04}. 
    \begin{figure}[htb]
      \centering
      \begin{tikzpicture}[->,>=stealth,shorten >=1pt,auto,node distance=2.3cm,
        state/.style={circle,minimum size=6mm,very thin,draw=black,initial text=}]

        \node[state,initial]   (0)    {$q_0$};
        \node[state] (3) [right of=0] {$q_3$};
        \node[state] (2) [above of=3] {$q_2$};
        \node[state] (1) [above of=2] {$q_1$};
        \node[state] (4) [below of=3] {$q_4$};
        \node[state] (5) [below of=4] {$q_5$};

        \node[state] (11) [above right of=1] {$p_1$};
        \node[state,accepting] (12) [right of=11] {$f_1$};
        \node[state] (13) [right of=1] {$s_1$};
        
        \node[state] (23) [right of=2] {$s_2$};
        \node[state] (21) [above right of=23] {$p_2$};
        \node[state] (24) [right of=23] {$s_3$};
        \node[state,accepting] (22) [above right of=24] {$f_2$};
        
        \node[state] (33) [right of=3] {$s_4$};
        \node[state] (31) [above right of=33] {$p_3$};
        \node[state] (34) [right of=33] {$s_5$};
        \node[state,accepting] (32) [above right of=34] {$f_3$};
        
        \node[state] (43) [right of=4] {$s_6$};
        \node[state] (44) [right of=43] {$s_7$};
        \node[state] (41) [above right of=44] {$p_4$};
        \node[state] (45) [right of=44] {$s_8$};
        \node[state,accepting] (42) [above right of=45] {$f_4$};
        
        \node[state] (51) [above right of=5] {$p_5$};
        \node[state] (55) [right of=5] {$s_9$};
        \node[state,accepting] (52) [above right of=55] {$f_5$};
        
        \path
          (0) edge[bend left] node {$a_1$} (1)
          (0) edge[bend left=20] node {$a_2$} (2)
          (0) edge node {$a_3$} (3)
          (0) edge[bend right=20] node {$a_4$} (4)
          (0) edge[bend right] node {$a_5$} (5)
          
          (1) edge node {$a$} (11)
          (11) edge node {$b_1$} (12)
          (1) edge node {$c_1$} (13)
          (13) edge node {$a$} (12)
          
          (2) edge[bend left=20] node {$a$} (21)
          (21) edge node {$b_2$} (22)
          (2) edge node {$c_1$} (23)
          (23) edge node {$c_2$} (24)
          (24) edge node {$a$} (22)
          
          (3) edge[bend left=20] node {$a$} (31)
          (31) edge node {$b_3$} (32)
          (3) edge node {$c_1$} (33)
          (33) edge node {$c_3$} (34)
          (34) edge node {$a$} (32)
          
          (4) edge[bend left=15] node {$a$} (41)
          (41) edge node {$b_4$} (42)
          (4) edge node {$c_2$} (43)
          (43) edge node {$c_3$} (44)
          (44) edge node {$c_4$} (45)
          (45) edge node {$a$} (42)
          
          (5) edge node {$a$} (51)
          (51) edge node {$b_5$} (52)
          (5) edge node {$c_4$} (55)
          (55) edge node {$a$} (52)
          ;
      \end{tikzpicture}
      \caption{The generator for language $K$ corresponding to the MINIMUM SET COVER instance $(S,C)$, where $S=\{b_1,b_2,b_3,b_4,b_5\}$ and $C=\{c_1=\{b_1,b_2,b_3\},$ $c_2=\{b_2,b_4\},c_3=\{b_3,b_4\},c_4=\{b_4,b_5\}\}$.}
      \label{fig04}
    \end{figure}
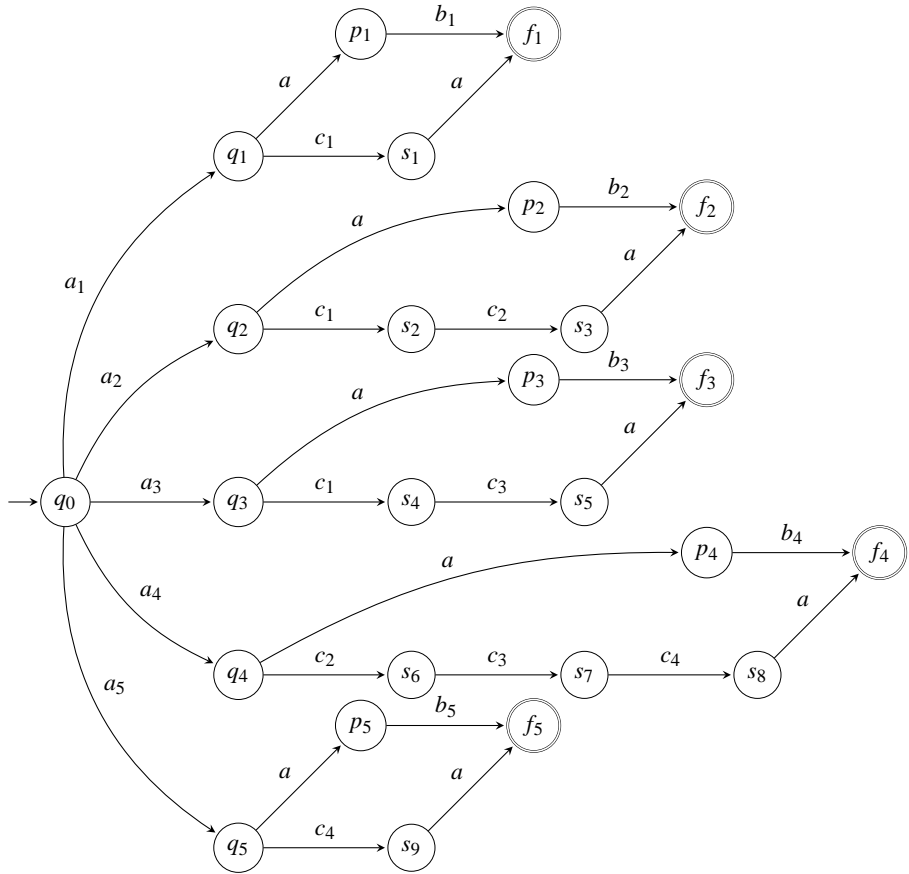
    Note that $\{c_1,c_4\}$ is the minimum set cover. Next, we define two event sets
    \begin{eqnarray*}
      & \Sigma_1 = S\cup \{a\}\cup\{a_i \mid i=1,2,\ldots,n\} \\ 
      & \text{ and } \\
      & 
      \Sigma_2 = C\cup \{a\}\cup\{a_i \mid i=1,2,\ldots,n\}\,.
    \end{eqnarray*}
    As the intersection $S\cap C$ is empty, it gives that the event set $\Sigma_s=\{a\}\cup\{a_i \mid i=1,2,\ldots,n\}$. We now prove that there exists a minimum set cover of cardinality at most $r$ if and only if there exists an extension of the event set $\Sigma_s$ of cardinality at most $r$ making the language $K$ conditionally decomposable. 
    
    Assume that there exists a minimum set cover $C'=\{c_{i_1},c_{i_2},\ldots,c_{i_r}\}\subseteq C$ of cardinality $r$. We prove that the language $K$ is conditionally decomposable with respect to $\Sigma_1$, $\Sigma_2$, and $\Sigma_k=\Sigma_s\cup\{c_{i_1},c_{i_2},\ldots,c_{i_r}\}$. The application of projection $P_{1+k}$ to language $K$ results in the language 
    \[
      P_{1+k}(K) = \sum_{i=1}^{n} (a_i a b_i + a_i P_{1+k}(c_{i_1} c_{i_2} \ldots c_{i_{b_i}}) a )\,,
    \] 
    and the application of projection $P_{2+k}$ to language $K$ results in the language 
    \[
      P_{2+k}(K) = \sum_{i=1}^{n} (a_i a + a_i c_{i_1} c_{i_2} \ldots c_{i_{b_i}} a)\,.
    \] 
    Note that the word $P_{1+k}(c_{i_1} c_{i_2} \ldots c_{i_{b_i}}) \in C'^*$ is nonempty because at least one set of the collection $C'$ covers the element $b_i$, for all $i=1,2,\ldots,n$. Let 
    \[
      X=C\setminus C'
    \]
    denote the complement of the collection $C'$, then the intersection $X\cap S$ is empty. As $C_{b_i}\cap C\ne\emptyset$, for each element $b_i$ of the set $S$, the language $P_{1+k}^{-1}P_{1+k}(c_{i_1} c_{i_2} \ldots c_{i_{b_i}})$ is not a subset of the language $X^*$. It can be seen that the intersection $S^* c_{i_1} S^* c_{i_2} S^* \ldots S^* c_{i_{b_i}} S^*\cap X^*=\emptyset$ is empty, that the intersection $P_{1+k}^{-1}P_{1+k}(c_{i_1} c_{i_2} \ldots c_{i_{b_i}})\cap S^*=\emptyset$ is empty, and that the intersection $P_{1+k}^{-1}P_{1+k}(c_{i_1} c_{i_2} \ldots c_{i_{b_i}})\cap S^* c_{i_1} S^* c_{i_2} S^* \ldots S^* c_{i_{b_i}} S^*=\{c_{i_1} c_{i_2} \ldots c_{i_{b_i}}\}$. Then the parallel composition of both projections of the language $K$,
    \begin{align*}
      & P_{1+k}(K)\parallel P_{2+k}(K)\\
      & = \sum_{i=1}^{n} (X^*a_i X^* a X^* b_i X^* + X^* a_i P_{1+k}^{-1}(P_{1+k}(c_{i_1} c_{i_2} \ldots c_{i_{b_i}})) a X^*) \nonumber\\
      & \cap \sum_{i=1}^{n} (S^* a_i S^* a S^* + S^* a_i S^* c_{i_1} S^* c_{i_2} S^* \ldots S^* c_{i_{b_i}} S^* a S^*)\\
      &= \sum_{i=1}^{n} (a_i a b_i + a_i c_{i_1} c_{i_2} \ldots c_{i_{b_i}} a )= K\nonumber\,,
    \end{align*}
    is equal to $K$.
    
    On the other hand, let $\Sigma_r\subseteq S\cup C$ be an extension of the event set $\Sigma_s$ of cardinality $r$ such that the language $K$ is conditionally decomposable with respect to event sets $\Sigma_1$, $\Sigma_2$, and $\Sigma_k=\Sigma_s\cup \Sigma_r$. Consider a symbol $b_i$ and two corresponding words $a_iab_i$ and $a_i c_{i_1} c_{i_2} \ldots c_{i_{b_i}} a$ from the language $K$. If $\Sigma_r\cap \{b_i, c_{i_1}, c_{i_2}, \ldots, c_{i_{b_i}}\}=\emptyset$, then the projections of these words to event sets $\Sigma_2\cup \Sigma_k$ and $\Sigma_1\cup \Sigma_k$ are, respectively, $P_{2+k}(a_iab_i)=a_ia$ and $P_{1+k}(a_i c_{i_1} c_{i_2} \ldots c_{i_{b_i}} a)=a_ia$. But then the word $a_i c_{i_1} c_{i_2} \ldots c_{i_{b_i}} a b_i\notin K$ belongs to $P_{1+k}(K)\parallel P_{2+k}(K)$, which is a contradiction. Hence, at least one of the symbols $b_i$, $c_{i_1}$, $c_{i_2}$, \ldots, $c_{i_{b_i}}$ must belong to the set $\Sigma_r$. In other words, at least one of these symbols covers the symbol $b_i$. We can now construct a covering $C'\subseteq C$ of cardinality at most $r$ as follows. For each $c$ in $\Sigma_r$, add the set $c$ to the covering $C'$, and for each $b$ in $\Sigma_r$, add any set $c$ from the set $C_b$ to the covering $C'$. It is then easy to see that the collection $C'$ covers the set $S$.
  \qed\end{proof}

  Note that an immediate consequence of the construction is that the minimal extension problem is NP-hard even for finite languages and two event sets.
  \begin{corollary}\label{NPhard}
    The minimal extension problem is NP-hard.
  \end{corollary}

  Similar minimal extension problems have been shown to be NP-hard in the literature, e.g., the minimal extension of observable event sets that guarantees observability of a language. However, unlike coobservability of decentralized control, conditional decomposability has an important property for large systems composed of many concurrent components---it can be checked in polynomial time in the number of components as shown in~\cite{SCL12}. In addition, an algorithm is presented there to compute an extension (but not necessarily the minimal one) of the shared event set such that the language under consideration becomes conditionally decomposable with respect to the original event sets $\Sigma_1$ and $\Sigma_2$ and the new (coordinator) event set $\Sigma_k$.

\section{Coordination control synthesis}\label{sec:controlsynthesis}
  In this section, we recall the coordination control problem and revise the necessary and sufficient conditions established in~\cite{ifacwc2011,scl2011,automatica2011} under which the problem is solvable. This revision leads to a simplification of existing notions and proofs, e.g., compare Definition~\ref{def:conditionalcontrollability} with \cite[Definition~9]{ifacwc2011} or the proof of Proposition~\ref{prop3} with the proof of \cite[Proposition~10]{ifacwc2011}. 
  
  We now summarize the results of this section compared to the existing results.
  The coordination control problem for non-prefix-closed languages was formulated in~\cite[Problem~7]{ifacwc2011}. The contribution of this paper is a simplification of the problem statement, namely, the prefix-closed part of the closed-loop system with a coordinator is shown to be a consequence of the non-prefix-closed case (see the note below the problem statement).
  The original definition of conditional controllability is simplified in Definition~\ref{def:conditionalcontrollability}.
  A simplified proof of Proposition~\ref{prop3} is presented.
  Proposition~\ref{prop4} is new.
  Theorem~\ref{th:controlsynthesissafety} is a simplified version of Theorem~18 stated in~\cite{ifacwc2011} without proof.

  \begin{problem}[Coordination control problem]\label{problem:controlsynthesis}
    Consider generators $G_1$ and $G_2$ over $\Sigma_1$ and $\Sigma_2$, respectively, and a generator $G_k$ (called a {\em coordinator\/}) over $\Sigma_k$. Assume that generators $G_1$ and $G_2$ are conditionally independent with respect to coordinator $G_k$, and that a specification $K \subseteq L_m(G_1 \| G_2 \| G_k)$ and its prefix-closure $\overline{K}$ are conditionally decomposable with respect to event sets $\Sigma_1$, $\Sigma_2$, and $\Sigma_k$. The aim of the coordination control synthesis is to determine nonblocking supervisors $S_1$, $S_2$, and $S_k$ for respective generators such that 
    \begin{align*}
      L_m(S_k/G_k)\subseteq P_k(K) && \text{ and } && L_m(S_i/ [G_i \parallel (S_k/G_k) ])\subseteq P_{i+k}(K),\, i=1,2\,,
    \end{align*}
    and the closed-loop system with the coordinator satisfies
    \begin{align*}
      L_m(S_1/ [G_1 \parallel (S_k/G_k) ]) ~ \parallel ~ L_m(S_2/ [G_2 \parallel (S_k/G_k) ]) & = K\,.
    \end{align*}
    $\hfill\diamond$
  \end{problem}

  One could expect that the equality $L(S_1/ [G_1 \parallel (S_k/G_k) ]) \parallel L(S_2/ [G_2 \parallel (S_k/G_k) ])=\overline{K}$ for prefix-closed languages should also be required in the statement of the problem. However, it is really sufficient to require only the equality for marked languages since it then implies that the equality $L(S_1/ [G_1 \parallel (S_k/G_k) ]) \parallel L(S_2/ [G_2 \parallel (S_k/G_k) ])=\overline{K}$ holds true because
  \begin{align*}
    \overline{K} &  = \overline{L_m(S_1/ [G_1 \parallel (S_k/G_k) ]) \parallel L_m(S_2/ [G_2 \parallel (S_k/G_k) ])}\\ 
                 &  \subseteq \overline{L_m(S_1/ [G_1 \parallel (S_k/G_k) ])} \parallel \overline{L_m(S_2/ [G_2 \parallel (S_k/G_k) ])}\\
                 &  \subseteq \overline{P_{1+k}(K)} \parallel \overline{P_{2+k}(K)} \\
                 &  = \overline{K}\,.
  \end{align*}
  Moreover, if such supervisors exist, their synchronous product is a nonblocking supervisor for the global plant, cf.~\cite{ifacwc2011}.

  Note that several conditions are required in the statement of the problem, namely, (i) the generators are conditionally independent with respect to the coordinator and (ii) the specification and its prefix-closure are conditionally decomposable with respect to event sets $\Sigma_1$, $\Sigma_2$, and $\Sigma_k$. These conditions can easily be fulfilled by the choice of an appropriate coordinator event set $\Sigma_k$. The reader is referred to~\cite{SCL12} for a polynomial algorithm extending a given event set so that the language becomes conditionally decomposable.
  
  In the statement of the problem, we have mentioned the notion of a coordinator. The fundamental question is the construction of such a coordinator. We now discuss one of the possible constructions of a suitable coordinator, which has already been discussed in the literature~\cite{ifacwc2011,scl2011,automatica2011}. We recall it here for the completeness.
  
  \begin{algorithm}[Construction of a coordinator]\label{algorithm}
    Consider generators $G_1$ and $G_2$ over $\Sigma_1$ and $\Sigma_2$, respectively, and let $K$ be a specification. Construct an event set $\Sigma_k$ and a coordinator $G_k$ as follows:
    \begin{enumerate}
      \item Set $\Sigma_k = \Sigma_1\cap \Sigma_2$ to be the set of all shared events.
      \item Extend $\Sigma_k$ so that $K$ and $\overline{K}$ are conditional decomposable, for instance using a method described in~\cite{SCL12}.
      \item Let the coordinator $G_k = P_k(G_1) \parallel P_k(G_2)$.
    \end{enumerate}
  \end{algorithm}
  
  So far, the only known condition ensuring that the projected generator is smaller than the original one is the observer property. Therefore, we might need to add step (2b) to extend the event set $\Sigma_k$ so that the projection $P_k$ is an $L(G_i)$-observer, for $i=1,2$, cf. Definition~\ref{def:observer} below. 

  Note that if we generalize this approach to more than two subsystems, the set $\Sigma_k$ of step $1$ is replaced with the set $\Sigma_s$ of all shared events defined in Section~\ref{sec:nphard} above.
  
  \begin{definition}[Observer]\label{def:observer}
    The projection $P_k:\Sigma^* \to \Sigma_k^*$, where $\Sigma_k$ is a subset of $\Sigma$, is an {\em $L$-observer\/} for a language $L$ over $\Sigma$ if, for all words $t$ in $P_k(L)$ and $s$ in $\overline{L}$, the word $P_k(s)$ is a prefix of $t$ implies that there exists a word $u$ in $\Sigma^*$ such that $su$ is in $L$ and $P_k(su)=t$.
  \end{definition}

  For a generator $G$ with $n$ states, the time and space complexity of the verification whether a projection $P$ is an $L(G)$-observer is $O(n^2)$, see~\cite{pcl08,pcl12}. An algorithm extending the event set to satisfy the property runs in time $O(n^3)$ and linear space. The most significant consequence of the observer property is the following theorem.
  
  \begin{theorem}[\cite{wong98}]
    If a projection $P$ is an $L(G)$-observer, for a generator $G$, then the minimal generator for the language $P(L(G))$ has no more states than the generator $G$.
  \end{theorem}

  This is an important result because it guarantees that the coordinator computed in Algorithm~\ref{algorithm} is smaller than the plant whenever the projection $P_k$ is an $L(G_1)\parallel L(G_2)$-observer.

\subsection{Conditional controllability}
  The concept of conditional controllability introduced in~\cite{KvS08} and later studied in~\cite{ifacwc2011,scl2011,automatica2011} plays the central role in the coordination control approach. In this paper, we revise and simplify this notion. In what follows, we use the notation $\Sigma_{i,u}=\Sigma_i\cap \Sigma_u$ to denote the set of locally uncontrollable events of the event set $\Sigma_i$.

  \begin{definition}\label{def:conditionalcontrollability}
    Let $G_1$ and $G_2$ be generators over $\Sigma_1$ and $\Sigma_2$, respectively, and let $G_k$ be a coordinator over $\Sigma_k$. A language $K\subseteq L(G_1\| G_2\| G_k)$ is {\em conditionally controllable\/} for generators $G_1$, $G_2$, $G_k$ and uncontrollable event sets $\Sigma_{1,u}$, $\Sigma_{2,u}$, $\Sigma_{k,u}$ if
    \begin{enumerate}
      \item\label{cc1} $P_k(K)$ is controllable with respect to $L(G_k)$ and $\Sigma_{k,u}$,
      \item\label{cc2} $P_{1+k}(K)$ is controllable with respect to $L(G_1) \parallel \overline{P_k(K)}$ and $\Sigma_{1+k,u}$,
      \item\label{cc3} $P_{2+k}(K)$ is controllable with respect to $L(G_2) \parallel \overline{P_k(K)}$ and $\Sigma_{2+k,u}$,
    \end{enumerate}
    where $\Sigma_{i+k,u}=(\Sigma_i\cup \Sigma_k)\cap \Sigma_u$, for $i=1,2$.
  \end{definition}
  
  The difference between Definition~\ref{def:conditionalcontrollability} and the definition in previous papers is that in item 2 we write $L(G_1) \parallel \overline{P_k(K)}$ instead of $L(G_1) \parallel \overline{P_k(K)} \parallel P_k^{2+k} (L(G_2) \| \overline{P_k(K)})$. This is possible because the assumption $K\subseteq L(G_1\| G_2\| G_k)$ implies the inclusion $\overline{P_k(K)}\subseteq (P^{k}_{k\cap 2})^{-1}P^2_{k\cap 2}(L(G_2))$, which results in the equality 
  \begin{align*}
    \overline{P_k(K)} \| P_k^{2+k} (L(G_2) \| \overline{P_k(K)}) & = \overline{P_k(K)} \| P_{k\cap 2}^{2} (L(G_2))\\
      & =\overline{P_k(K)}\cap (P^{k}_{k\cap 2})^{-1}P^2_{k\cap 2}(L(G_2))\\
      & =\overline{P_k(K)}
  \end{align*}
  by Lemma~\ref{lemma:Wonham} (see the Appendix).
  Hence we have the following.
  
  \begin{lemma}
    Definition~\ref{def:conditionalcontrollability} and \cite[Definition~9]{ifacwc2011} of conditional controllability are equivalent.
  \end{lemma}

  The following proposition demonstrates that every conditionally controllable and conditionally decomposable language is controllable. 
  \begin{proposition}\label{prop3}
    Let $G_i$ be a generator over $\Sigma_i$, for $i=1,2,k$, and let $G=G_1\| G_2\| G_k$. Let $K\subseteq L_m(G)$ be such a specification that the language $\overline{K}$ is conditionally decomposable with respect to event sets $\Sigma_1$, $\Sigma_2$, $\Sigma_k$, and conditionally controllable for generators $G_1$, $G_2$, $G_k$ and uncontrollable event sets $\Sigma_{1,u}$, $\Sigma_{2,u}$, $\Sigma_{k,u}$. Then the language $K$ is controllable with respect to the plant language $L(G)$ and uncontrollable event set $\Sigma_u=\Sigma_{1,u}\cup\Sigma_{2,u}$.
  \end{proposition}
  \begin{proof}
    Since the language $\overline{P_{1+k}(K)}$ is controllable with respect to $L(G_1) \parallel \overline{P_k(K)}$ and $\Sigma_{1+k,u}$, and $\overline{P_{2+k}(K)}$ is controllable with respect to $L(G_2) \parallel \overline{P_k(K)}$ and $\Sigma_{2+k,u}$, Lemma~\ref{feng} implies that the language $\overline{K} = \overline{P_{1+k}(K)} \parallel \overline{P_{2+k}(K)}$ is controllable with respect to $L(G_1) \parallel \overline{P_k(K)} \parallel L(G_2) \parallel \overline{P_k(K)} = L(G) \parallel \overline{P_k(K)}$ and $\Sigma_u$, where the equality is by commutativity of the synchronous product and by the fact that $\overline{P_k(K)}\subseteq L(G_k)$. As the language $\overline{P_k(K)}$ is controllable with respect to $L(G_k)$ and $\Sigma_{k,u}$, by Definition~\ref{def:conditionalcontrollability}, the language $L(G) \parallel \overline{P_k(K)}$ is controllable with respect to $L(G) \parallel L(G_k) = L(G)$ by Lemma~\ref{feng}. Finally, by Lemma~\ref{lem_trans}, $\overline{K}$ is controllable with respect to $L(G)$ and $\Sigma_u$, which means that $K$ is controllable with respect to $L(G)$ and $\Sigma_u$.
  \qed\end{proof}
  
  On the other hand, controllability does not imply conditional controllability.
  \begin{example}
    Let $G$ be a generator such that $L(G)=\overline{\{au\}}\parallel \overline{\{bu\}}=\overline{\{abu,bau\}}$. Then the language $K=\{a\}$ is controllable with respect to $L(G)$ and $\Sigma_u=\{u\}$. Moreover, both languages $K$ and $\overline{K}$ are conditionally decomposable with respect to event sets $\{a,u\}$, $\{b,u\}$, and $\Sigma_k=\{u\}$, but the language $P_k(K)=\{\varepsilon\}$ is not controllable with respect to $L(G_k)=P_k(L(G))=\{u\}$ and $\Sigma_{k,u}=\{u\}$.
  \hfill$\triangleleft$\end{example}

  However, we show below that if the observer property and local control consistency (LCC) are satisfied, the previous implication holds. To prove this, we need the following definition of LCC. Note that unlike our previous papers, we use a weaker notion of local control consistency (LCC) presented in~\cite{SB11} instead of output control consistency (OCC). 

  \begin{definition}[LCC]
    Let $L$ be a prefix-closed language over $\Sigma$, and let $\Sigma_0$ be a subset of $\Sigma$. The projection $P_0:\Sigma^*\to \Sigma_0^*$ is {\em locally control consistent\/} (LCC) with respect to a word $s\in L$ if for all events $\sigma_u\in \Sigma_0\cap \Sigma_u$ such that $P_0(s)\sigma_u\in P_0(L)$, it holds that either there does not exist any word $u\in (\Sigma\setminus \Sigma_0)^*$ such that $su\sigma_u \in L$, or there exists a word $u\in (\Sigma_u\setminus \Sigma_0)^*$ such that $su\sigma_u \in L$. The projection $P_0$ is LCC with respect to a language $L$ if $P_0$ is LCC for all words of $L$.
  \end{definition}

  Now the opposite implication to the one proven in Proposition~\ref{prop3} can be stated.

  \begin{proposition}\label{prop4}
    Let $L$ be a prefix-closed language over $\Sigma$, and let $K\subseteq L$ be a language that is controllable with respect to $L$ and $\Sigma_u$. If, for $i\in\{k,1+k,2+k\}$, the projection $P_i$ is an $L$-observer and LCC for $L$, then the language $K$ is conditionally controllable.
  \end{proposition}
  \begin{proof}
    Let $s\in \overline{P_k(K)}$, $a\in \Sigma_{k,u}$, and $sa \in P_k(L)$. Then there exists a word $w$ in $\overline{K}$ such that $P_k(w)=s$. By the observer property, there exists a word $u$ in $(\Sigma\setminus \Sigma_k)^*$ such that $wua\in L$ and $P_k(wua)=sa$. By LCC, there exists another word $u'$ in $(\Sigma_u\setminus \Sigma_k)^*$ such that $wu'a\in L$, that is, $wu'a$ is in $\overline{K}$ by controllability. Hence, $sa\in \overline{P_k(K)}$.
    
    Let $s\in \overline{P_{1+k}(K)}$, $a\in \Sigma_{1+k,u}$, and $sa\in L(G_1)\parallel \overline{P_k(K)}$. Then there exists a word $w$ in $\overline{K}$ such that $P_{1+k}(w)=s$. By the observer property, there exists a word $u$ in $(\Sigma\setminus \Sigma_{1+k})^*$ such that $wua \in L$ and $P_{1+k}(wua)=sa$. By LCC, there exists another word $u'$ in $(\Sigma_u\setminus \Sigma_{1+k})^*$ such that $wu'a\in L$, that is, $wu'a$ is in $\overline{K}$ by controllability. Hence, $sa\in \overline{P_{1+k}(K)}$.

    The proof for the case of $k+2$ is similar to that of $k+1$.
  \qed\end{proof}

\subsection{Conditionally closed languages}
  In this subsection we turn our attention to general specification languages that need not be prefix-closed.
  Analogously to the notion of $L_m(G)$-closed languages, we recall the notion of conditionally-closed languages defined in~\cite{ifacwc2011}.

  \begin{definition}\label{def:conditionalclosed}
    A nonempty language $K$ over $\Sigma$ is {\em conditionally closed\/} for generators $G_1$, $G_2$, $G_k$ if
    \begin{enumerate}
      \item\label{ccl1} $P_k(K)$ is $L_m(G_k)$-closed,
      \item\label{ccl2} $P_{1+k}(K)$ is $L_m(G_1) \parallel P_k(K)$-closed,
      \item\label{ccl3} $P_{2+k}(K)$ is $L_m(G_2) \parallel P_k(K)$-closed.
    \end{enumerate}
  \end{definition}

  If a language $K$ is conditionally closed and conditionally controllable, then there exists a nonblocking supervisor $S_k$ such that $L_m(S_k/G_k)=P_k(K)$, which follows from the basic theorem of supervisory control applied to languages $P_k(K)$ and $L(G_k)$, see~\cite{CL08}.
  
  As noted in \cite[page 164]{CL08}, if $K\subseteq L_m(G)$ is $L_m(G)$-closed, then so is the supremal controllable sublanguage of $K$. However, this does not imply that the language $P_k(K)$ is $L_m(G_k)$-closed, for any generator $G=G_1\| G_2\| G_k$ such that the coordinator $G_k$ makes generators $G_1$ and $G_2$ conditionally independent.
  \begin{example}\label{ex2}
    Let the event sets be $\Sigma_{1} = \{a_1,a\}$, $\Sigma_{2}=\{a_2,a\}$, and $\Sigma_k=\{a\}$, respectively, and let the specification language be $K=\{a_1a_2a,a_2a_1a\}$. Then the application of projections results in languages $P_{1+k}(K) = \{a_1a\}$, $P_{2+k}(K) = \{a_2a\}$, and $P_{k}(K) = \{a\}$, and the language $K = P_{1+k}(K)\parallel P_{2+k}(K)$ is conditionally decomposable. Define generators $G_1$, $G_2$, $G_k$ so that $L_m(G_1)=P_{1+k}(K)$, $L_m(G_2)=P_{2+k}(K)$, and $L_m(G_k)=\overline{P_{k}(K)}=\{\eps,a\}$. Then $L_m(G)=K$ and the language $K$ is $L_m(G)$-closed. However, the language $P_k(K)\subset \overline{P_k(K)}$ is not $L_m(G_k)$-closed.
    \hfill$\triangleleft$
  \end{example}

\subsection{Existence of supervisors}
  The following theorem is a revised version (based on the simplification of conditional controllability, Definition~\ref{def:conditionalcontrollability}) of a result presented without proof in~\cite{ifacwc2011}.
  
  \begin{theorem}\label{th:controlsynthesissafety}
    Consider the setting of Problem~\ref{problem:controlsynthesis}. There exist nonblocking supervisors $S_1$, $S_2$, $S_k$ such that
    \begin{align}\label{eq:controlsynthesissafety}
      L_m(S_1/[G_1 \parallel (S_k/G_k)]) \parallel L_m(S_2/[G_2 \parallel (S_k/G_k)]) =  K
    \end{align}
    if and only if the specification language $K$ is both conditionally controllable with respect to generators $G_1$, $G_2$, $G_k$ and uncontrollable event sets $\Sigma_{1,u}$, $\Sigma_{2,u}$, $\Sigma_{k,u}$, and conditionally closed with respect to generators $G_1$, $G_2$, $G_k$.
  \end{theorem}
  \begin{proof}
    Let $K$ satisfy the assumptions, and let $G=G_1 \| G_2 \| G_k$ be the global plant. As the language $K$ is a subset of $L_m(G)$, its projection $P_k(K)$ is a subset of $L_m(G_k)$. By the assumption, the language $P_k(K)$ is $L_m(G_k)$-closed and controllable with respect to $L(G_k)$ and $\Sigma_{k,u}$. By the basic theorem of supervisory control~\cite{RW87} there exists a nonblocking supervisor $S_k$ such that $L_m(S_k/G_k) = P_k(K)$. As the language $P_{1+k}(K)$ is a subset of languages $L_m(G_1\| G_k)$ and $(P_{k}^{1+k})^{-1} P_k(K)$, we have that $P_{1+k}(K)$ is included in $L_m(G_1) \parallel P_k(K)$. These relations and the assumption that the system is conditionally controllable and conditionally closed imply the existence of a nonblocking supervisor $S_1$ such that $L_m(S_1/ [ G_1 \parallel (S_k/G_k) ]) = P_{1+k}(K)$. A similar argument shows that there exists a nonblocking supervisor $S_2$ such that $L_m(S_2/ [ G_2 \parallel (S_k/G_k) ]) = P_{2+k}(K)$. Since $K$ and $\overline{K}$ are conditionally decomposable, it follows that $L_m(S_1/[G_1 \parallel (S_k/G_k)]) \parallel L_m(S_2/[G_2 \parallel (S_k/G_k)]) = P_{1+k}(K) \parallel P_{2+k}(K)  = K$.

    To prove the converse implication, the projections $P_k$, $P_{1+k}$, $P_{2+k}$ are applied to (\ref{eq:controlsynthesissafety}), which can be rewritten as $K = L_m(S_1 \| G_1 \parallel S_2 \| G_2 \parallel S_k \| G_k)$. Thus, the projection $P_k(K) = P_k\left(L_m(S_1 \| G_1 \parallel S_2 \| G_2 \parallel S_k \| G_k)\right)$ is a subset of $L_m(S_k \| G_k) = L_m(S_k/G_k)$. On the other hand, $L_m(S_k/G_k) \subseteq P_k(K)$, cf. Problem~\ref{problem:controlsynthesis}. Hence, by the basic controllability theorem, the language $P_k(K)$ is both controllable with respect to $L(G_k)$ and $\Sigma_{k,u}$, and $L_m(G_k)$-closed. As $\Sigma_{1+k}\cap \Sigma_{2+k}=\Sigma_k$, the application of projection $P_{1+k}$ to (\ref{eq:controlsynthesissafety}) and assumptions of Problem~\ref{problem:controlsynthesis} give that $P_{1+k}(K) \subseteq L_m(S_1/ [ G_1 \parallel (S_k/G_k)]) \subseteq P_{1+k}(K)$. Taking $G_1 \| (S_k/G_k)$ as a new plant, we get from the basic supervisory control theorem that the language $P_{1+k}(K)$ is controllable with respect to $L(G_1 \| (S_k/G_k))$ and $\Sigma_{1+k,u}$, and that it is $L_m(G_1 \| (S_k/G_k))$-closed. The case of the language $P_{2+k}(K)$ is analogous.
  \qed\end{proof}

\section{Supremal conditionally controllable sublanguages}\label{sec:procedure}
  Necessary and sufficient conditions for the existence of nonblocking supervisors $S_1$, $S_2$, and $S_k$ that achieve a considered specification language using our coordination control architecture have been presented in Theorem~\ref{th:controlsynthesissafety}. However, in many cases control specifications fail to be conditionally controllable and, similarly as in the monolithic supervisory control, supremal conditionally controllable sublanguages should be investigated.

  Let $\supCC(K, L, (\Sigma_{1,u}, \Sigma_{2,u}, \Sigma_{k,u}))$ denote the supremal conditionally controllable sublanguage of $K$ with respect to $L=L(G_1\|G_2\|G_k)$ and sets of uncontrollable events $\Sigma_{1,u}$, $\Sigma_{2,u}$, $\Sigma_{k,u}$. The supremal conditionally controllable sublanguage always exists, cf.~\cite{scl2011} for the case of prefix-closed languages.
  \begin{theorem}\label{existence}
    The supremal conditionally controllable sublanguage of a given language $K$ always exists and is equal to the union of all conditionally controllable sublanguages of the language $K$.
  \end{theorem}
  \begin{proof}
    Let $I$ be an index set, and let $K_i$, for $i\in I$, be conditionally controllable sublanguages of $K\subseteq L(G_1\|G_2\|G_k)$. 
    To prove that the language $P_k(\cup_{i\in I} K_i)$ is controllable with respect to $L(G_k)$ and $\Sigma_{k,u}$, note that
    \begin{align*}
      P_k\left(\cup_{i\in I} \overline{K_i}\right)\Sigma_{k,u} \cap L(G_k)
        &= \cup_{i\in I} \left(P_k(\overline{K_i})\Sigma_{k,u} \cap L(G_k) \right) \\
        &\subseteq \cup_{i\in I} P_k(\overline{K_i}) \\
        &= P_k\left(\cup_{i\in I}\overline{K_i}\right),
    \end{align*}
    where the inclusion is by controllability of the language $P_k(K_i)$ with respect to $L(G_k)$ and $\Sigma_{k,u}$.
    Next, to prove that
    \begin{align*}
      P_{1+k}\left(\cup_{i\in I} \overline{K_i}\right)\Sigma_{1+k,u} 
        &\cap L(G_1)\parallel P_k\left(\cup_{i\in I} \overline{K_i}\right)
        \subseteq P_{1+k}\left(\cup_{i\in I} \overline{K_i}\right),
    \end{align*}
    note that
    \begin{align*}
      P_{1+k}&\left(\cup_{i\in I} \overline{K_i}\right)\Sigma_{1+k,u} \cap L(G_1)\parallel P_k\left(\cup_{i\in I} \overline{K_i}\right)\\
      &= \cup_{i\in I} \left( P_{1+k}(\overline{K_i})\Sigma_{1+k,u}\right) 
        \cap \cup_{i\in I} \left(L(G_1)\parallel P_k(\overline{K_i})\right)\\
      &= \cup_{i\in I}\cup_{j\in I} \left( P_{1+k}(\overline{K_i})\Sigma_{1+k,u} 
        \cap L(G_1)\parallel P_k(\overline{K_j})\right)\,.
    \end{align*}
    Consider two different indexes $i$ and $j$ from $I$ such that
    \begin{align*}
      P_{1+k}(\overline{K_i})\Sigma_{1+k,u} 
      \cap L(G_1)\parallel P_k(\overline{K_j})
      \not\subseteq P_{1+k}\left(\cup_{i\in I} \overline{K_i}\right).
    \end{align*}
    Then there exist a word $x$ in $P_{1+k}(\overline{K_i})$ and an uncontrollable event $u$ in $\Sigma_{1+k,u}$ such that $xu$ belongs to the language $L(G_1)\| P_k(\overline{K_j})$, and $xu$ does not belong to $P_{1+k}\left(\cup_{i\in I} \overline{K_i}\right)$. It follows that
      $P_k(x)$ belongs to $P_k(\overline{K_i})$ and
      $P_k(xu)$ belongs to $P_k(\overline{K_j})$.
    If $P_k(xu)$ belongs to $P_k(\overline{K_i})$, then $xu$ belongs to $L(G_1)\| P_k(\overline{K_i})$, and controllability of the language $P_{1+k}(\overline{K_i})$ with respect to $L(G_1)\| P_k(\overline{K_i})$ implies that $xu$ belongs to $P_{1+k}\left(\cup_{i\in I} \overline{K_i}\right)$; hence, $P_k(xu)$ does not belong to $P_k(\overline{K_i})$. If the event $u$ does not belong to $\Sigma_{k,u}$, then $P_k(xu)=P_k(x)$ belongs to $P_k(\overline{K_i})$, which is not the case. Thus, $u$ belongs to $\Sigma_{k,u}$. As $P_k(\overline{K_i})\cup P_k(\overline{K_j})$ is a subset of $L(G_k)$, we get that $P_k(xu) = P_k(x)u$ belongs to $L(G_k)$. However, controllability of the language $P_k(\overline{K_i})$ with respect to $L(G_k)$ and $\Sigma_{k,u}$ implies that the word $P_k(xu)$ belongs to $P_k(\overline{K_i})$. This is a contradiction.
    
    As the case for the projection $P_{2+k}$ is analogous, the proof is complete.
  \qed\end{proof}  
  
  Still, it is a difficult problem to compute a supremal conditional controllable sublanguage. Consider the setting of Problem~\ref{problem:controlsynthesis} and define the languages
  \begin{equation}\tag{*}\label{eq0}
    \boxed{
    \begin{aligned}
      \supC_k     & =  \supC(P_k(K), L(G_k), \Sigma_{k,u})\\
      \supC_{1+k} & =  \supC(P_{1+k}(K), L(G_1) \parallel \overline{\supC_k}, \Sigma_{1+k,u})\\
      \supC_{2+k} & =  \supC(P_{2+k}(K), L(G_2) \parallel \overline{\supC_k}, \Sigma_{2+k,u})
    \end{aligned}}
  \end{equation}

  Interestingly, the following inclusion always holds.
  \begin{lemma}\label{lemma16}
    Consider the setting of Problem~\ref{problem:controlsynthesis}, and languages defined in (\ref{eq0}). Then the language $P_k(\supC_{i+k})$ is a subset of the language $\supC_k$, for $i=1,2$.
  \end{lemma}
  \begin{proof}
    By definition, the language $P_k(\supC_{i+k})$ is a subset of languages $\overline{\supC_k}$ and $P_k(K)$. To prove that $P_k(\supC_{i+k})$ is a subset of $\supC_k$, we prove that the language $\overline{\supC_k}\cap P_k(K)$ is a subset of $\supC_k$. To do this, it is sufficient to show that the language $\overline{\supC_k}\cap P_k(K)$ is controllable with respect to $L(G_k)$ and $\Sigma_{k,u}$. 
    
    Thus, consider a word $s$ in $\overline{\overline{\supC_k}\cap P_k(K)}$, an uncontrollable event $u$ in $\Sigma_{k,u}$, and the word $su$ in $L(G_k)$. By controllability of $\supC_k$, the word $su$ belongs to $\overline{\supC_k}$, which is a subset of $\overline{P_k(K)}$. That is, there exists a word $v$ such that $suv$ is in $\supC_k$, which is a subset of $P_k(K)$. This means that the word $suv$ belongs to $\overline{\supC_k}\cap P_k(K)$, which implies that the word $su$ is in $\overline{\overline{\supC_k}\cap P_k(K)}$. This completes the proof.
  \qed\end{proof}

  It turns out that if the converse inclusion also holds, then we immediately obtain the supremal conditionally-controllable sublanguage.
  \begin{theorem}\label{thm2}
    Consider the setting of Problem~\ref{problem:controlsynthesis}, and languages defined in~(\ref{eq0}). If $\supC_k$ is a subset of $P_k(\supC_{i+k})$, for $i=1,2$, then 
    \[
      \supC_{1+k} \parallel \supC_{2+k} = \supCC(K, L, (\Sigma_{1,u}, \Sigma_{2,u}, \Sigma_{k,u}))\,.
    \]
  \end{theorem}
  \begin{proof}
    Let $\supCC = \supCC(K, L, (\Sigma_{1,u}, \Sigma_{2,u}, \Sigma_{k,u}))$ and $M = \supC_{1+k} \parallel \supC_{2+k}$.
    To prove that $M$ is a subset of $\supCC$, we show that (i) $M$ is a subset of $K$ and (ii) $M$ is conditionally controllable with respect to generators $G_1$, $G_2$, $G_k$ and uncontrollable event sets $\Sigma_{1,u}$, $\Sigma_{2,u}$, $\Sigma_{k,u}$. To this aim, notice that $M$ is a subset of $P_{1+k}(K) \parallel P_{2+k}(K) = K$, because $K$ is conditionally decomposable. Moreover, by Lemmas~\ref{lemma:Wonham} and~\ref{lemma16}, the language $P_k(M) = P_k(\supC_{1+k}) \cap P_k(\supC_{2+k})=\supC_k$, which is controllable with respect to $L(G_k)$ and $\Sigma_{k,u}$. Similarly, $P_{i+k}(M) = \supC_{i+k} \parallel P_k(\supC_{j+k}) = \supC_{i+k} \parallel \supC_{k} = \supC_{i+k}$, for $j\neq i$, which is controllable with respect to $L(G_i)\parallel \overline{P_k(M)}$. Hence, $M$ is a subset of $\supCC$.

    To prove the opposite inclusion, it is sufficient, by Lemma~\ref{lem11}, to show that the language $P_{i+k}(\supCC)$ is a subset of $\supC_{i+k}$, for $i=1,2$. To prove this note that the language $P_{1+k}(\supCC)$ is controllable with respect to $L(G_1)\parallel \overline{P_k(\supCC)}$ and $\Sigma_{1+k,u}$, and the language $L(G_1) \parallel \overline{P_k(\supCC)}$ is controllable with respect to $L(G_1) \parallel \overline{\supC_k}$ and $\Sigma_{1+k,u}$ by Lemma~\ref{feng}, because the language $P_k(\supCC)$ being controllable with respect to $L(G_k)$ implies that it is also controllable with respect to $\overline{\supC_k}$, which is a subset of $L(G_k)$. By Lemma~\ref{lem_trans}, the language $P_{1+k}(\supCC)$ is controllable with respect to $L(G_1) \parallel \overline{\supC_k}$ and $\Sigma_{1+k,u}$, which implies that $P_{1+k}(\supCC)$ is a subset of $\supC_{1+k}$. The other case is analogous. Hence, the language $\supCC$ is a subset of $M$ and the proof is complete.
  \qed\end{proof}

  \begin{example}
    This example demonstrates that the language $\supC_k$ is not always included in the language $P_k(\supC_{i+k})$. Moreover, it does not hold even if projections are observers or satisfy the LCC property. 
    
    Consider systems $G_1$ and $G_2$ shown in Fig.~\ref{figEx}, and the specification $K$ as shown in Fig.~\ref{figExx}.
    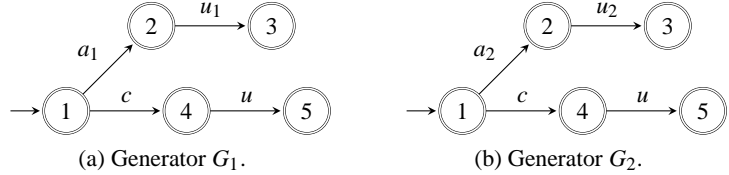
\begin{figure}[htb]
      \centering
      \subfloat[Generator $G_1$.]{
        \begin{tikzpicture}[->,>=stealth,shorten >=1pt,auto,node distance=1.6cm,
          state/.style={circle,minimum size=6mm, very thin,draw=black,initial text=}]
          \node[state,initial,accepting] (1) {1};
          \node[state,accepting] (2) [above right of=1] {2};
          \node[state,accepting] (3) [right of=2] {3};
          \node[state,accepting] (4) [right of=1] {4};
          \node[state,accepting] (5) [right of=4] {5};
          \path
            (1) edge node {$a_1$} (2)
            (2) edge node {$u_1$} (3)
            (1) edge node {$c$} (4)
            (4) edge node {$u$} (5);
        \end{tikzpicture}
      }
      \qquad
      \subfloat[Generator $G_2$.]{
        \begin{tikzpicture}[->,>=stealth,shorten >=1pt,auto,node distance=1.6cm,
          state/.style={circle,minimum size=6mm, very thin,draw=black,initial text=}]
          \node[state,initial,accepting] (1) {1};
          \node[state,accepting] (2) [above right of=1] {2};
          \node[state,accepting] (3) [right of=2] {3};
          \node[state,accepting] (4) [right of=1] {4};
          \node[state,accepting] (5) [right of=4] {5};
          \path
            (1) edge node {$a_2$} (2)
            (2) edge node {$u_2$} (3)
            (1) edge node {$c$} (4)
            (4) edge node {$u$} (5);
        \end{tikzpicture}
      }
      \caption{Generators $G_1$ and $G_2$.}
      \label{figEx}
    \end{figure}
    \begin{figure}[htb]
      \centering
      \begin{tikzpicture}[->,>=stealth,shorten >=1pt,auto,node distance=1.6cm,
        state/.style={circle,minimum size=6mm, very thin,draw=black,initial text=}]
        \node[state,initial,accepting] (1) {1};
        \node[state,accepting] (2) [above right of=1] {2};
        \node[state,accepting] (3) [right of=2] {3};
        \node[state,accepting] (4) [right of=3] {4};
        \node[state,accepting] (5) [right of=1] {5};
        \node[state,accepting] (6) [right of=5] {6};
        \node[state,accepting] (7) [right of=6] {7};
        \path
          (1) edge node {$a_1$} (2)
          (2) edge node {$a_2$} (3)
          (3) edge node {$u_2$} (4)
          (1) edge node {$a_2$} (5)
          (5) edge node {$a_1$} (6)
          (6) edge node {$u_1$} (7);
      \end{tikzpicture}
      \caption{Specification $K$.}
      \label{figExx}
    \end{figure}
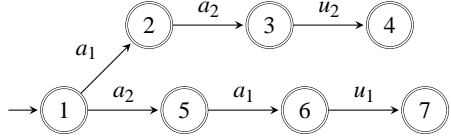
    Controllable events are $\Sigma_c=\{a_1,a_2,c\}$, and coordinator events are $\Sigma_k=\{a_1,a_2,c,u\}$. Construct the coordinator $G_k=P_k(G_1)\parallel P_k(G_2)$. It can be verified that $K$ is conditionally decomposable, $\supC_k=\overline{\{a_1a_2,a_2a_1\}}$, $\supC_{1+k}=\overline{\{a_2a_1u_1\}}$, and $\supC_{2+k}=\overline{\{a_1a_2u_2\}}$. Hence, $\supC_k$ is not a subset of $P_k(\supC_{i+k})$.
    
    It can be verified that projections $P_k$, $P_{1+k}$, $P_{2+k}$ are $L(G_1\|G_2)$-observers and LCC for the language $L(G_1\|G_2)$.
  \hfill$\triangleleft$
  \end{example}

  Recall that it is still open how to compute the supremal conditionally-controllable sublanguage for a general, non-prefix-closed language. Consider the example above and note that the words $a_1a_2$ and $a_2a_1$ from $\supC_k$ do not appear in the projection of the supremal conditionally-controllable sublanguage, that is, no words with both letters $a_1$ and $a_2$ appear in the supremal conditionally-controllable sublanguage. Thus, we can remove these words from $\supC_k$ (basically from the coordinator) and recompute the supremal controllable sublanguage (denoted by $\supC_k'$), that is, 
  \[
    \supC_k'=\supC(\cap_{i=1,2}P_k(\supC_{i+k}),L(G_k),\Sigma_{k,u})=\{\eps\}
  \]
  and, similarly, recompute $\supC_{i+k}$ using $\supC_k'$ instead of $\supC_k$. Note that the plant is changed because the coordinator restricts it more than before. An application of Theorem~\ref{thm2} could thus be as follows. If $\supC_k\not\subseteq P_k(\supC_{i+k})$, then the natural approach seems to be to remove from $\supC_k$ all words violating the inclusion, and to recompute $\supC_{i+k}$, for $i=1,2$, with respect to this new $\supC_k'$, that is
  \begin{equation}\tag{**}\label{eq1}
    \boxed{
    \begin{aligned}
      \supC_k'     & =  \supC(P_k(\supC_{1+k})\cap P_k(\supC_{2+k}), L(G_k), \Sigma_{k,u})\\
      \supC_{1+k}' & =  \supC(\supC_{1+k}, L(G_1) \parallel \overline{\supC_k'}, \Sigma_{1+k,u})\\
      \supC_{2+k}' & =  \supC(\supC_{2+k}, L(G_2) \parallel \overline{\supC_k'}, \Sigma_{2+k,u})
    \end{aligned}}
  \end{equation}
  In our example, we get that $\supC_{1+k}'=\{\eps\}$ and $\supC_{2+k}'=\{\eps\}$ satisfy the assumption that $\supC_k'\subseteq P_k(\supC_{i+k}')$, for $i=1,2$, hence Theorem~\ref{thm2} applies. It is not yet clear whether this method can be used in general, namely whether it always terminates and the result is the supremal conditionally-controllable sublanguage. It is only known that if it terminates, the result is conditionally controllable (see the end of Section~\ref{sec:nonblocking} for more discussion). Another problem is that it requires to compute the projection, which can be exponential in general, because the observer property is not ensured. One of the natural investigations of this problem is to work with nondeterministic representations. Several attempts in this direction were done in the literature although they usually handle the case where only the plant is nondeterministic, while the specification is deterministic, see, e.g., \cite{su:2010,Su:2012}. Even more, it is a question how to test the inclusion from Theorem~\ref{thm2}.
  
  Finally, if $\supC_{i+k}$ and $\supC_k'$ are nonconflicting, the language $\supC_{i+k}\|\supC_k'$ is controllable with respect to $L(G_i)\|\overline{\supC_k}\|\overline{\supC_k'}=L(G_i)\|\overline{\supC_k'}$ by Lemma~\ref{feng}. This observation gives the following result for prefix-closed languages.
  
  \begin{lemma}\label{lem00}
    Let $K=\overline{K}\subseteq L=L(G_1\| G_2\| G_k)$, where $G_i$ is a generator over $\Sigma_i$, for $i=1,2,k$. Assume that $K$ is conditionally decomposable, and define the languages $\supC_k$, $\supC_{1+k}$ and $\supC_{2+k}$ as in (\ref{eq0}). If $\supC_k\not\subseteq P_k(\supC_{i+k})$, for $i\in\{1,2\}$, define the language $\supC_k'$ as in (\ref{eq1}). Then the language
    \[
      \supC_{1+k} \parallel \supC_{2+k} \parallel \supC_k'
    \]
    is conditionally controllable with respect to $G_1,G_2,G_k$ and $\Sigma_{1,u},\Sigma_{2,u},\Sigma_{k,u}$.
  \end{lemma}

  Note that if we have any specification $K$, which is conditionally decomposable, then the specification $K\parallel L$ is also conditionally decomposable. The opposite is not true.
  \begin{lemma}
    Let $K$ be conditionally decomposable with respect to event sets $\Sigma_1$, $\Sigma_2$, $\Sigma_k$, and let $L=L_1\parallel L_2\parallel L_k$, where $L_i$ is over $\Sigma_i$, for $i=1,2,k$. Then the language $K\parallel L$ is conditionally decomposable with respect to event sets $\Sigma_1$, $\Sigma_2$, $\Sigma_k$.
  \end{lemma}
  \begin{proof}
    By the assumption we have that $K=P_{1+k}(K) \| P_{2+k}(K)$. Then 
    \begin{align*}
      K\| L &= P_{1+k}(K) \| P_{2+k}(K) \| L_1 \| L_2 \| L_k\\
            &= P_{1+k}(K) \| L_1 \| L_k \parallel P_{2+k}(K) \| L_2 \| L_k\\
            &= P_{1+k}(K \| L_1 \| L_k) \parallel P_{2+k}(K \| L_2 \| L_k)
    \end{align*}
    where the last equality is by Lemma~\ref{lemma:Wonham}.
    By Lemma~\ref{CDlemma}, $K\|L$ is conditionally decomposable with respect to event sets $\Sigma_1$, $\Sigma_2$, and $\Sigma_k$.
    \qed
  \end{proof}

  \begin{example}
    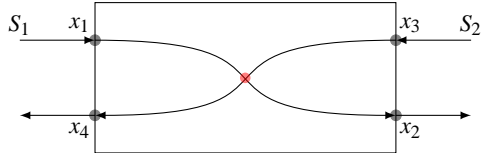
\begin{figure}[htb]
      \centering
      \begin{tikzpicture}
        \filldraw [gray]  (-2,-.5) circle (2pt) 
                          (2,.5) circle (2pt) 
                          (-2,.5) circle (2pt) 
                          (2,-.5) circle (2pt);
        \clip (-3.5,-1) rectangle (3.5,1);
        \draw (-2,-1) rectangle (2,1);
        \draw [<-,>=latex,auto,name path=curve 1] (-2,-.5) .. controls (1,-.5) and (-1,.5) .. (2,.5);
        \draw [->,>=latex,auto,name path=curve 2] (-2,.5) .. controls (1,.5) and (-1,-.5) .. (2,-.5);
        \fill [name intersections={of=curve 1 and curve 2, name=i, total=\t}]
              [red, opacity=0.5, every node/.style={above left, black, opacity=1}]
        \foreach \s in {1,...,\t}{(i-\s) circle (2pt)};
        \draw [<-,>=latex,auto] (-3,-.5) -- (-2,-.5);
        \draw [->,>=latex,auto] (-3,.5) -- (-2,.5);
        \draw (-3,.7) node {$S_1$};
        \draw [<-,>=latex,auto] (3,-.5) -- (2,-.5);
        \draw [->,>=latex,auto] (3,.5) -- (2,.5);
        \draw (3,.7) node {$S_2$};
        \draw (-2.2,.7) node {$x_1$};
        \draw (2.2,.7) node {$x_3$};
        \draw (-2.2,-.7) node {$x_4$};
        \draw (2.2,-.7) node {$x_2$};
      \end{tikzpicture}
      \caption{A railway crossroad}\label{figA1}
    \end{figure}
    Consider a situation at a railway station. There are several tracks that cross each other at some points. Obviously, the traffic has to be controlled at those points. For simplicity, we consider only two one-way tracks that cross at some point, that is, trains going from west to east use track one, while trains going from east to west use track two. The traffic is controlled by traffic lights. 
    
    Thus, consider the railway crossroad with two traffic lights, $S_1$ and $S_2$, and two entry points $x_1,x_3$ and two exit points $x_2,x_4$, as depicted in Fig.~\ref{figA1}. Each traffic light has values $g_i$ (green) and $r_i$ (red), for $i=1,2$. Colors of the traffic lights are controllable. The plant is then given as a parallel composition of two systems $G_1$ and $G_2$ depicted in Fig.~\ref{figGens}. For safety reasons, each system is able to set the traffic light to red at any moment. It can set the traffic light to green and the trains are detected entering ($x_1$ or $x_3$) and leaving ($x_2$ or $x_4$) the crossroad.
    \begin{figure}[htb]
      \centering
      \begin{tikzpicture}[->,>=stealth,shorten >=1pt,auto,node distance=1.4cm,
        state/.style={circle,minimum size=6mm,very thin,draw=black,initial text=}]
        \node[state,initial,accepting] (1) {1};
        \node[state] (2) [right of=1] {2};
        \path
          (1) edge[loop above] node {$r_1,x_2$} (1)
          (1) edge[bend left] node {$g_1$} (2)
          (2) edge[bend left] node {$r_1$} (1)
          (2) edge[loop above] node {$x_1,x_2$} (2);
      \end{tikzpicture}\hspace{2cm}
      \begin{tikzpicture}[->,>=stealth,shorten >=1pt,auto,node distance=1.4cm,
        state/.style={circle,minimum size=6mm,very thin,draw=black,initial text=}]
        \node[state,initial,accepting] (1) {1};
        \node[state] (2) [right of=1] {2};
        \path
          (1) edge[loop above] node {$r_2,x_4$} (1)
          (1) edge[bend left] node {$g_2$} (2)
          (2) edge[bend left] node {$r_2$} (1)
          (2) edge[loop above] node {$x_3,x_4$} (2);
      \end{tikzpicture}
      \caption{Generators $G_1$ and $G_2$}\label{figGens}
    \end{figure}
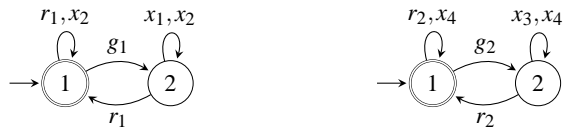
    
    To define the specification, it is natural that a train is allowed to enter the crossroad only if its traffic light is green. The purpose of the entry and exit points $x_i$, $i=1,2,3,4$, is to allow a limited number of trains in the crossroad area from the direction of the green light. The light can turn red at any moment, but the other traffic light can be set to green only if all the trains have left the crossroad area. In this example, we consider the case where at most three trains are allowed to enter the crossroad area on one green light. For this purpose, the entry points must also be controllable to protect another train to enter. This part of the specification is modeled by buffers depicted in Fig.~\ref{figBufs}.
    \begin{figure}[htb]
      \centering
      \begin{tikzpicture}[->,>=stealth,shorten >=1pt,auto,node distance=1.4cm,
        state/.style={circle,minimum size=6mm,very thin,draw=black,initial text=}]
        \node[state,initial,accepting] (1) {1};
        \node[state,accepting] (2) [right of=1] {2};
        \node[state,accepting] (3) [right of=2] {3};
        \node[state,accepting] (4) [right of=3] {4};
        \path
          (1) edge[loop above] node {$r_1,g_2,x_2$} (1)
          (1) edge[bend left] node {$x_1$} (2)
          (2) edge[bend left] node {$x_2$} (1)
          (2) edge[bend left] node {$x_1$} (3)
          (3) edge[bend left] node {$x_2$} (2)
          (3) edge[bend left] node {$x_1$} (4)
          (4) edge[bend left] node {$x_2$} (3)
          (2) edge[loop above] node {$r_1$} (2)
          (3) edge[loop above] node {$r_1$} (3)
          (4) edge[loop above] node {$r_1$} (4);
      \end{tikzpicture}\hspace{1cm}
      \begin{tikzpicture}[->,>=stealth,shorten >=1pt,auto,node distance=1.4cm,
        state/.style={circle,minimum size=6mm,very thin,draw=black,initial text=}]
        \node[state,initial,accepting] (1) {1};
        \node[state,accepting] (2) [right of=1] {2};
        \node[state,accepting] (3) [right of=2] {3};
        \node[state,accepting] (4) [right of=3] {4};
        \path
          (1) edge[loop above] node {$r_2,g_1,x_4$} (1)
          (1) edge[bend left] node {$x_3$} (2)
          (2) edge[bend left] node {$x_4$} (1)
          (2) edge[bend left] node {$x_3$} (3)
          (3) edge[bend left] node {$x_4$} (2)
          (3) edge[bend left] node {$x_3$} (4)
          (4) edge[bend left] node {$x_4$} (3)
          (2) edge[loop above] node {$r_2$} (2)
          (3) edge[loop above] node {$r_2$} (3)
          (4) edge[loop above] node {$r_2$} (4);
      \end{tikzpicture}
      \caption{The two buffers}\label{figBufs}
    \end{figure}
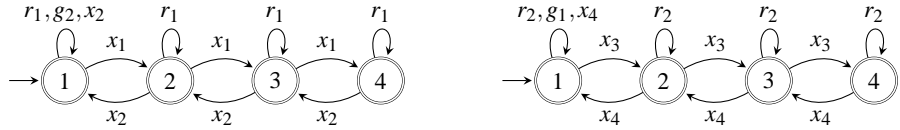
    Another part of the specification governs the behavior of the traffic lights. First, both lights must be red before one of the traffic lights is set to green, stay green for a while, and then must be set to red again. The traffic lights should take turns, so that no trains are waiting for ever, see Fig.~\ref{figSpec1}. For simplicity, we do not model the mechanism (such as a clock) that sets the traffic lights to green for a specific amount of time units. The overall specification is then depicted in Fig.~\ref{figSpec}.
    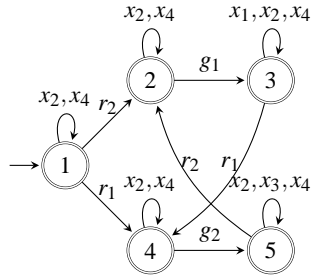
\begin{figure}[htb]
      \centering
      \begin{tikzpicture}[->,>=stealth,shorten >=1pt,auto,node distance=1.6cm,
        state/.style={circle,minimum size=6mm,very thin,draw=black,initial text=}]
        \node[state,initial,accepting] (1) {1};
        \node[state,accepting] (2) [above right of=1] {2};
        \node[state,accepting] (3) [right of=2] {3};
        \node[state,accepting] (4) [below right of=1] {4};
        \node[state,accepting] (5) [right of=4] {5};
        \path
          (1) edge node[above] {$r_2$} (2)
          (1) edge node[above] {$r_1$} (4)
          (2) edge node {$g_1$} (3)
          (4) edge node {$g_2$} (5)
          (3) edge[bend left=20] node[above] {$r_1$} (4)
          (5) edge[bend left=20] node[above] {$r_2$} (2)
          (3) edge[loop above] node {$x_1,x_2,x_4$} (3)
          (5) edge[loop above] node {$x_2,x_3,x_4$} (5)
          (1) edge[loop above] node {$x_2,x_4$} (1)
          (2) edge[loop above] node {$x_2,x_4$} (2)
          (4) edge[loop above] node {$x_2,x_4$} (4);
      \end{tikzpicture}
      \caption{The traffic lights' part of the specification}\label{figSpec1}
    \end{figure}
    \begin{figure}[htb]
      \centering
      \begin{tikzpicture}[->,>=stealth,shorten >=1pt,auto,node distance=1.8cm,
        state/.style={circle,minimum size=6mm,very thin,draw=black,initial text=}]
        \node[state,initial,accepting] (1) {1};
        \node[state,accepting] (2) [above right of=1] {2};
        \node[state,accepting] (3) [above right of=2] {3};
        \node[state,accepting] (4) [below right of=1] {4};
        \node[state,accepting] (5) [below right of=4] {5};
        \path
          (1) edge node[above] {$r_2$} (2)
          (1) edge[loop above] node {$x_2,x_4$} (1)
          (1) edge node {$r_1$} (4)
          (2) edge node[above] {$g_1$} (3)
          (2) edge[loop above] node {$x_2,x_4$} (2)
          (4) edge node {$g_2$} (5)
          (4) edge[loop below] node {$x_2,x_4$} (4)
          (3) edge[bend left=20] node {$r_1$} (4)
          (5) edge[bend right=20] node[right] {$r_2$} (2)
          (3) edge[loop above] node {$x_2,x_4$} (3)
          (5) edge[loop below] node {$x_2,x_4$} (5);
          
        \node[state,accepting] (6) [right of=3] {6};
        \node[state,accepting] (7) [right of=6] {7};
        \node[state,accepting] (8) [right of=7] {8};
        \path
          (3) edge[bend left] node {$x_1$} (6)
          (6) edge[bend left] node {$x_2$} (3)
          (6) edge[bend left] node {$x_1$} (7)
          (7) edge[bend left] node {$x_2$} (6)
          (7) edge[bend left] node {$x_1$} (8)
          (8) edge[bend left] node {$x_2$} (7)
          (6) edge[loop above] node {$x_4$} (6)
          (7) edge[loop above] node {$x_4$} (7)
          (8) edge[loop above] node {$x_4$} (8);
          
        \node[state,accepting] (9)  [right of=5] {9};
        \node[state,accepting] (10) [right of=9] {10};
        \node[state,accepting] (11) [right of=10] {11};
        \path
          (5) edge[bend left] node {$x_3$} (9)
          (9) edge[bend left] node {$x_4$} (5)
          (9) edge[bend left] node {$x_3$} (10)
          (10) edge[bend left] node {$x_4$} (9)
          (10) edge[bend left] node {$x_3$} (11)
          (11) edge[bend left] node {$x_4$} (10)
          (9) edge[loop below] node {$x_2$} (9)
          (10) edge[loop below] node {$x_2$} (10)
          (11) edge[loop below] node {$x_2$} (11);
          
        \node[state,accepting] (12) [below of=6,node distance=1.5cm] {12};
        \node[state,accepting] (13) [below of=7,node distance=1.5cm] {13};
        \node[state,accepting] (14) [below of=8,node distance=1.5cm] {14};
        \path
          (6) edge node {$r_1$} (12)
          (7) edge node {$r_1$} (13)
          (8) edge node {$r_1$} (14)
          (12) edge[loop below] node {$x_4$} (12)
          (13) edge[loop below] node {$x_4$} (13)
          (14) edge[loop below] node {$x_4$} (14)
          (14) edge node {$x_2$} (13)
          (13) edge node {$x_2$} (12)
          (12) edge[bend left=20] node {$x_2$} (4);
          
        \node[state,accepting] (15) [above of=9,node distance=1.5cm] {15};
        \node[state,accepting] (16) [above of=10,node distance=1.5cm] {16};
        \node[state,accepting] (17) [above of=11,node distance=1.5cm] {17};
        \path
          (9) edge node {$r_2$} (15)
          (10) edge node {$r_2$} (16)
          (11) edge node {$r_2$} (17)
          (15) edge[loop above] node {$x_2$} (15)
          (16) edge[loop above] node {$x_2$} (16)
          (17) edge[loop above] node {$x_2$} (17)
          (17) edge node {$x_4$} (16)
          (16) edge node {$x_4$} (15)
          (15) edge[bend right=20] node[right] {$x_4$} (2);
      \end{tikzpicture}
      \caption{The overall specification}\label{figSpec}
    \end{figure}
    The set of uncontrollable events is thus $\Sigma_u = \{x_2, x_4\}$; all other events are controllable.
    
    To make the specification controllable with respect to $\Sigma_1$, $\Sigma_2$, and $\Sigma_k$ (where $\Sigma_k$ is initialized to the empty set), we need to take $\Sigma_k=\{g_1, g_2, r_1\}$. Now we can compute the coordinator as the projection $P_k(G_1)\|P_k(G_2)$, and the languages $\supC_k$, $\supC_{1+k}$ and $\supC_{2+k}$ as defined in~(\ref{eq0}), see Figs.~\ref{figSups1}, \ref{figSups2}, and~\ref{figSups3}. 
    \begin{figure}[htb]
      \centering
      \begin{tikzpicture}[->,>=stealth,shorten >=1pt,auto,node distance=1.6cm,
        state/.style={circle,minimum size=6mm,very thin,draw=black,initial text=}]
        \node[state,initial,accepting] (1) {1};
        \node[state,accepting] (2) [right of=1] {2};
        \node[state,accepting] (3) [above right of=2] {3};
        \node[state] (4) [below right of=3] {4};
        \path
          (1) edge node {$r_1$} (2)
          (1) edge[bend right=40] node {$g_1$} (4)
          (2) edge[bend left] node {$g_2$} (3)
          (3) edge[bend left] node {$g_1$} (4)
          (4) edge node[above] {$r_1$} (2);
      \end{tikzpicture}
      \caption{Supervisor $\supC_k$}\label{figSups1}
    \end{figure}
    It can be verified that $\supC_k\subseteq P_k(\supC_{i+k})$, for $i=1,2$, hence Theorem~\ref{thm2} applies and the result (that is, in the monolithic notation, the language $\supC_{1+k}\|\supC_{2+k}$) is the supremal conditionally-con\-troll\-able sublanguage of the specification, cf. Fig.~\ref{figSupCC}. Note that the difference with the specification is the correct marking of the states. 
    \begin{figure}[htb]
      \centering
      \begin{tikzpicture}[->,>=stealth,shorten >=1pt,auto,node distance=1.6cm,
        state/.style={circle,minimum size=6mm,very thin,draw=black,initial text=}]
        \node[state,initial,accepting] (1) {1};
        \node[state,accepting] (2) [below right of=1] {2};
        \node[state,accepting] (3) [right of=2] {3};
        \node[state] (4) [above right of=3] {4};
        \path
          (1) edge node {$r_1$} (2)
          (1) edge[loop above] node {$x_2$} (1)
          (2) edge[loop below] node {$x_2$} (2)
          (3) edge[loop right] node {$x_2$} (3)
          (1) edge node {$g_1$} (4)
          (2) edge node {$g_2$} (3)
          (4) edge[bend right=10] node[above] {$r_1$} (2)
          (4) edge[loop above] node {$x_2$} (4)
          (3) edge node[above] {$g_1$} (4);
          
        \node[state] (6) [right of=4] {5};
        \node[state] (7) [right of=6] {6};
        \node[state] (8) [right of=7] {7};
        \path
          (4) edge[bend left] node {$x_1$} (6)
          (6) edge[bend left] node {$x_2$} (4)
          (6) edge[bend left] node {$x_1$} (7)
          (7) edge[bend left] node {$x_2$} (6)
          (7) edge[bend left] node {$x_1$} (8)
          (8) edge[bend left] node {$x_2$} (7);
          
        \node[state,accepting] (12) [below of=6,node distance=1.5cm] {8};
        \node[state,accepting] (13) [below of=7,node distance=1.5cm] {9};
        \node[state,accepting] (14) [below of=8,node distance=1.5cm] {10};
        \path
          (6) edge node {$r_1$} (12)
          (7) edge node {$r_1$} (13)
          (8) edge node {$r_1$} (14)
          (14) edge node {$x_2$} (13)
          (13) edge node {$x_2$} (12)
          (12) edge[bend left=20] node {$x_2$} (2);
      \end{tikzpicture}
      \caption{Supervisor $\supC_{1+k}$}\label{figSups2}
    \end{figure}
    \begin{figure}[htb]
      \centering
      \begin{tikzpicture}[->,>=stealth,shorten >=1pt,auto,node distance=1.5cm,
        state/.style={circle,minimum size=6mm,very thin,draw=black,initial text=}]
        \node[state,initial,accepting] (1) {1};
        \node[state,accepting] (2) [below right of=1] {2};
        \node[state] (3) [above right of=2] {3};
        \node[state,accepting] (4) [right of=3] {4};
        \node[state] (5) [right of=4] {5};
        \path
          (1) edge node {$r_2$} (2)
          (1) edge[loop above] node {$x_4$} (1)
          (2) edge[loop below] node {$x_4$} (2)
          (3) edge[loop left] node {$x_4$} (3)
          (4) edge[loop above] node {$x_4$} (4)
          (1) edge[bend left] node {$r_1$} (4)
          (2) edge node {$g_1$} (3)
          (4) edge node {$g_2$} (5)
          (3) edge node {$r_1$} (4)
          (5) edge[bend left=20] node[above] {$r_2$} (2)
          (5) edge[loop above] node {$x_4$} (5);
          
        \node[state] (9)  [right of=5] {6};
        \node[state] (10) [right of=9] {7};
        \node[state] (11) [right of=10] {8};
        \path
          (5) edge[bend left] node {$x_3$} (9)
          (9) edge[bend left] node {$x_4$} (5)
          (9) edge[bend left] node {$x_3$} (10)
          (10) edge[bend left] node {$x_4$} (9)
          (10) edge[bend left] node {$x_3$} (11)
          (11) edge[bend left] node {$x_4$} (10);
          
        \node[state,accepting] (15) [below of=9,node distance=1.5cm] {9};
        \node[state,accepting] (16) [below of=10,node distance=1.5cm] {10};
        \node[state,accepting] (17) [below of=11,node distance=1.5cm] {11};
        \path
          (9) edge node {$r_2$} (15)
          (10) edge node {$r_2$} (16)
          (11) edge node {$r_2$} (17)
          (17) edge node {$x_4$} (16)
          (16) edge node {$x_4$} (15)
          (15) edge[bend left=20] node[above] {$x_4$} (2);
      \end{tikzpicture}
      \caption{Supervisor $\supC_{2+k}$}\label{figSups3}
    \end{figure}
    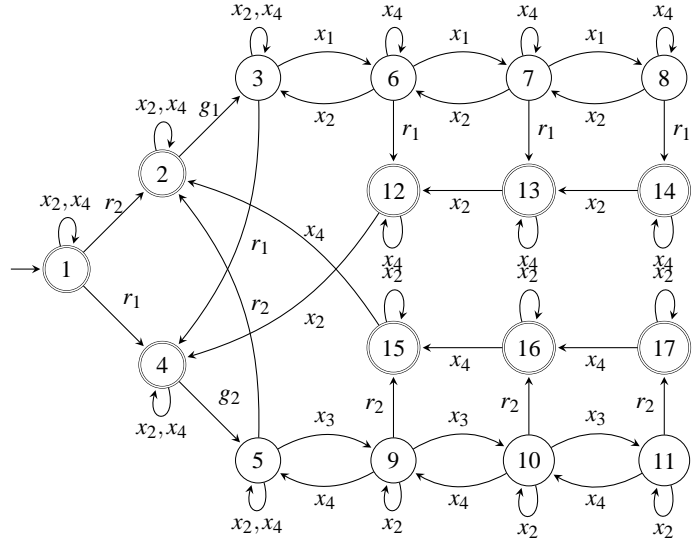
\begin{figure}[htb]
      \centering
      \begin{tikzpicture}[->,>=stealth,shorten >=1pt,auto,node distance=1.8cm,
        state/.style={circle,minimum size=6mm,very thin,draw=black,initial text=}]
        \node[state,initial,accepting] (1) {1};
        \node[state,accepting] (2) [above right of=1] {2};
        \node[state] (3) [above right of=2] {3};
        \node[state,accepting] (4) [below right of=1] {4};
        \node[state] (5) [below right of=4] {5};
        \path
          (1) edge node[above] {$r_2$} (2)
          (1) edge[loop above] node {$x_2,x_4$} (1)
          (1) edge node {$r_1$} (4)
          (2) edge node[above] {$g_1$} (3)
          (2) edge[loop above] node {$x_2,x_4$} (2)
          (4) edge node {$g_2$} (5)
          (4) edge[loop below] node {$x_2,x_4$} (4)
          (3) edge[bend left=20] node {$r_1$} (4)
          (5) edge[bend right=20] node[right] {$r_2$} (2)
          (3) edge[loop above] node {$x_2,x_4$} (3)
          (5) edge[loop below] node {$x_2,x_4$} (5);
          
        \node[state] (6) [right of=3] {6};
        \node[state] (7) [right of=6] {7};
        \node[state] (8) [right of=7] {8};
        \path
          (3) edge[bend left] node {$x_1$} (6)
          (6) edge[bend left] node {$x_2$} (3)
          (6) edge[bend left] node {$x_1$} (7)
          (7) edge[bend left] node {$x_2$} (6)
          (7) edge[bend left] node {$x_1$} (8)
          (8) edge[bend left] node {$x_2$} (7)
          (6) edge[loop above] node {$x_4$} (6)
          (7) edge[loop above] node {$x_4$} (7)
          (8) edge[loop above] node {$x_4$} (8);
          
        \node[state] (9)  [right of=5] {9};
        \node[state] (10) [right of=9] {10};
        \node[state] (11) [right of=10] {11};
        \path
          (5) edge[bend left] node {$x_3$} (9)
          (9) edge[bend left] node {$x_4$} (5)
          (9) edge[bend left] node {$x_3$} (10)
          (10) edge[bend left] node {$x_4$} (9)
          (10) edge[bend left] node {$x_3$} (11)
          (11) edge[bend left] node {$x_4$} (10)
          (9) edge[loop below] node {$x_2$} (9)
          (10) edge[loop below] node {$x_2$} (10)
          (11) edge[loop below] node {$x_2$} (11);
          
        \node[state,accepting] (12) [below of=6,node distance=1.5cm] {12};
        \node[state,accepting] (13) [below of=7,node distance=1.5cm] {13};
        \node[state,accepting] (14) [below of=8,node distance=1.5cm] {14};
        \path
          (6) edge node {$r_1$} (12)
          (7) edge node {$r_1$} (13)
          (8) edge node {$r_1$} (14)
          (12) edge[loop below] node {$x_4$} (12)
          (13) edge[loop below] node {$x_4$} (13)
          (14) edge[loop below] node {$x_4$} (14)
          (14) edge node {$x_2$} (13)
          (13) edge node {$x_2$} (12)
          (12) edge[bend left=20] node {$x_2$} (4);
          
        \node[state,accepting] (15) [above of=9,node distance=1.5cm] {15};
        \node[state,accepting] (16) [above of=10,node distance=1.5cm] {16};
        \node[state,accepting] (17) [above of=11,node distance=1.5cm] {17};
        \path
          (9) edge node {$r_2$} (15)
          (10) edge node {$r_2$} (16)
          (11) edge node {$r_2$} (17)
          (15) edge[loop above] node {$x_2$} (15)
          (16) edge[loop above] node {$x_2$} (16)
          (17) edge[loop above] node {$x_2$} (17)
          (17) edge node {$x_4$} (16)
          (16) edge node {$x_4$} (15)
          (15) edge[bend right=20] node[right] {$x_4$} (2);
      \end{tikzpicture}
      \caption{The supremal conditionally-controllable sublanguage $\supC_{1+k}\|\supC_{2+k}$}\label{figSupCC}
    \end{figure}
  \end{example}

\section{Coordinator for nonblockingness}\label{sec:nonblocking}
   So far, we have only considered a coordinator for safety. In this section, we discuss a coordinator for nonblockingness. To this end, we first prove a fundamental theoretical result and then give an algorithm to construct a coordinator for nonblockingness. 
   
   Recall that a generator $G$ is nonblocking if $\overline{L_m(G)}=L(G)$.
  \begin{theorem}
    Consider languages $L_1$ over $\Sigma_1$ and $L_2$ over $\Sigma_2$, and let the projection $P_0:(\Sigma_1\cup \Sigma_2)^*\to \Sigma_0^*$, with $\Sigma_1\cap \Sigma_2\subseteq \Sigma_0$, be an $L_i$-observer, for $i=1,2$. Let $G_0$ be a nonblocking generator with $L_m(G_0)=P_0(L_1)\| P_0(L_2)$. Then the composed language $L_1\| L_2\| L_m(G_0)$ is nonblocking, that is, $\overline{L_1\| L_2\| L_m(G_0)} = \overline{L_1}\| \overline{L_2}\| \overline{L_m(G_0)}$.
  \end{theorem}
  \begin{proof}
    Let $L_0=L_m(G_0)$. By Lemma~\ref{fengT41}, 
    $
      \overline{L_1\| L_2\| L_0} = \overline{L_1} \| \overline{L_2} \| \overline{L_0}
    $
    if and only if 
    \[
      \overline{P_0(L_1) \| P_0(L_2) \| L_0} = \overline{P_0(L_1)} \| \overline{P_0(L_2)} \| \overline{L_0}\,.
    \] 
    However, for our choice of the coordinator, this equality always holds because both sides of the later equation are $\overline{L_0}$.
  \qed\end{proof}

  This result is demonstrated in the following example.
  \begin{example}\label{ex3}
    \begin{figure}[htb]
    \centering
    \begin{tikzpicture}[->,>=stealth,shorten >=1pt,auto,node distance=1.4cm,
      state/.style={circle,minimum size=6mm,very thin,draw=black,initial text=}]
      \node[state,initial,accepting] (1) {1};
      \node[state,accepting] (2) [right of=1] {2};
      \node[state] (3) [right of=2] {3};
      \node[state,accepting] (4) [right of=3] {4};
      \node[] (5) [right of=4] {};
      
      \node[state,initial,accepting] (11) [right of=5] {1};
      \node[state] (12) [above right of=11] {2};
      \node[state,accepting] (13) [below right of=12] {3};
      \node[state] (14) [below right of=11] {4};
      
      \path
      (1) edge node {$a$} (2)
      (2) edge node {$b$} (3)
      (3) edge node {$d$} (4)
      (11) edge[bend left] node {$a$} (12)
      (12) edge[bend left] node {$c$} (13)
      (11) edge[bend right] node {$d$} (14)
      (14) edge[bend right] node {$a$} (13);
    \end{tikzpicture}
    \caption{Generators $G_1$ and $G_2$}\label{fig4}
    \end{figure}
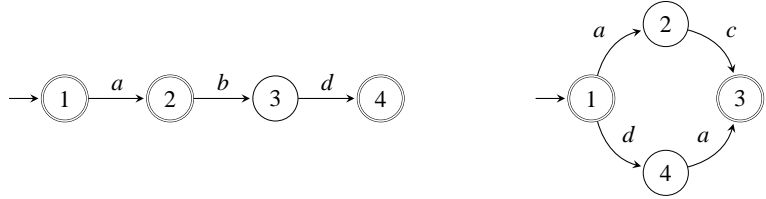
    Consider two nonblocking generators $G_1$ and $G_2$ depicted in Fig.~\ref{fig4}. Their synchronous product is shown in Fig.~\ref{fig5}. One can see that the generator $G_1\|G_2$ is blocking because no marked state is reachable from state 3. It can be verified that the projection $P:\{a,b,c,d\}^*\to \{a,b,d\}^*$ is an $L(G_1)$- and $L(G_2)$-observer. The generator $G_0$ is then a nonblocking (trimmed) part of the synchronous product $P(G_1)\|P(G_2)$ of generators depicted in Fig.~\ref{fig6}, that is $L_m(G_0)=\{a\}$, and the synchronous product of $G_1\|G_2$ with $G_0$ is shown in Fig.~\ref{fig7}. One can see that the result is nonblocking. It is important to notice that event $b$ belongs to the event set of the generator $G_0$.
    \begin{figure}[htb]
    \centering
    \begin{tikzpicture}[->,>=stealth,shorten >=1pt,auto,node distance=1.8cm,
      state/.style={circle,minimum size=6mm,very thin,draw=black,initial text=}]

      \node[state,initial,accepting] (0) {0};
      \node[state] (11) [right of=0] {1};
      \node[state] (12) [above right of=11] {2};
      \node[state] (13) [below right of=12] {3};
      \node[state,accepting] (14) [below right of=11] {4};
      
      \path
      (0)  edge node {$a$} (11)
      (11) edge[bend left] node {$b$} (12)
      (12) edge[bend left] node {$c$} (13)
      (11) edge[bend right] node {$c$} (14)
      (14) edge[bend right] node {$b$} (13);
    \end{tikzpicture}
    \caption{Synchronous product $G_1\|G_2$}\label{fig5}
    \end{figure}
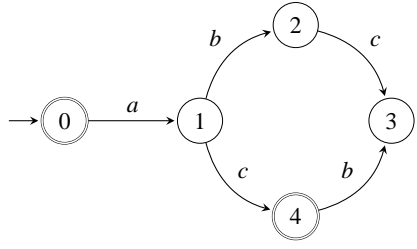
    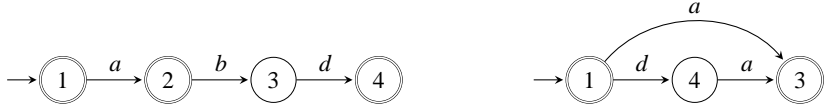
\begin{figure}[htb]
    \centering
    \begin{tikzpicture}[->,>=stealth,shorten >=1pt,auto,node distance=1.4cm,
      state/.style={circle,minimum size=6mm,very thin,draw=black,initial text=}]
      \node[state,initial,accepting] (1) {1};
      \node[state,accepting] (2) [right of=1] {2};
      \node[state] (3) [right of=2] {3};
      \node[state,accepting] (4) [right of=3] {4};
      \node[] (5) [right of=4] {};
      
      \node[state,initial,accepting] (11) [right of=5] {1};
      \node[state] (14) [right of=11] {4};
      \node[state,accepting] (13) [right of=14] {3};
      
      \path
      (1) edge node {$a$} (2)
      (2) edge node {$b$} (3)
      (3) edge node {$d$} (4)
      (11) edge[bend left=50] node {$a$} (13)
      (11) edge node {$d$} (14)
      (14) edge node {$a$} (13);
    \end{tikzpicture}
    \caption{Generators $P(G_1)$ and $P(G_2)$}\label{fig6}
    \end{figure}
    \begin{figure}[htb]
    \centering
    \begin{tikzpicture}[->,>=stealth,shorten >=1pt,auto,node distance=1.8cm,
      state/.style={circle,minimum size=6mm,very thin,draw=black,initial text=}]

      \node[state,initial,accepting] (0) {0};
      \node[state] (11) [right of=0] {1};
      \node[state,accepting] (13) [right of=11] {2};
      
      \path
      (0)  edge node {$a$} (11)
      (11) edge node {$c$} (13);
    \end{tikzpicture}
    \caption{Synchronous product $G_1\|G_2\|G_0$}\label{fig7}
    \end{figure}
  \end{example}
  
  The previous example shows that even thought the result is nonblocking, it is disputable whether such a coordinator is acceptable. If we assume that event $b$ is uncontrollable, then the coordinator prevents an uncontrollable event from happening and the result depicted in Fig.~\ref{fig7} is not controllable with respect to the plant depicted in Fig.~\ref{fig5}. Although it is not explicitly stated that a coordinator is not allowed to do so, we further discuss this issue and suggest a solution useful in our coordination control framework.

  In general, local supervisors $\supC_{1+k}$ and $\supC_{2+k}$ computed in Section~\ref{sec:procedure} might be blocking. However, we can always choose the language
  \begin{equation}\label{eq2}
    L_C=\supC(P_0(\supC_{1+k})\parallel P_0(\supC_{2+k}),\ \overline{P_0(\supC_{1+k})}\parallel \overline{P_0(\supC_{2+k})},\ \Sigma_{0,u})\,,
  \end{equation}
  where the projection $P_0$ is a $\supC_{i+k}$-observer, for $i=1,2$. The following result shows that the language $\supC_{1+k}\|\supC_{2+k}\|L_C$ is nonblocking and controllable.
  
  \begin{theorem}\label{thm22}
    Consider the notation as defined in Problem~\ref{problem:controlsynthesis}, Algorithm~\ref{algorithm}, (\ref{eq0}), and (\ref{eq2}). Then the language
    \[
      \overline{\supC_{1+k}\parallel \supC_{2+k}\parallel L_C} = \overline{\supC_{1+k}} \parallel \overline{\supC_{2+k}} \parallel \overline{L_C}
    \]
    is controllable with respect to the plant language $L(G_1)\|L(G_2)$.
  \end{theorem}
  \begin{proof}
    To prove nonblockingness, we use Lemma~\ref{fengT41} in two steps. Namely, it holds that $\overline{\supC_{i+k}\|L_C}=\overline{\supC_{i+k}}\|\overline{L_C}$ if and only if $\overline{P_0(\supC_{i+k})\|L_C}=\overline{P_0(\supC_{i+k})}\|\overline{L_C}$, for $i=1,2$, which always holds because both sides of the later equation are equal to $\overline{L_C}$. Using Lemma~\ref{fengT41} again,
    \begin{align*}
        \overline{\supC_{1+k}\|L_C \parallel \supC_{2+k}\|L_C}
        = \overline{\supC_{1+k}\|L_C} \parallel \overline{\supC_{2+k}\|L_C}
    \end{align*}
    if and only if 
    \begin{align}\label{eq3}
        \overline{P_0(\supC_{1+k}\|L_C) \parallel P_0(\supC_{2+k}\|L_C)}
        = \overline{P_0(\supC_{1+k}\|L_C)} \parallel \overline{P_0(\supC_{2+k}\|L_C)}
    \end{align}
    because if the projection $P_0$ is a $\supC_{i+k}$-observer, for $i=1,2$, and an $L_C$-observer (since it is an identity), then the projection $P_0$ is also an $\supC_{i+k}\| L_C$-observer by~\cite{pcl06}. But (\ref{eq3}) always holds because $P_0(\supC_{i+k}\|L_C)=P_0(\supC_{i+k})\|L_C=L_C$, by Lemma~\ref{lemma:Wonham}, hence both sides are equal to $\overline{L_C}$. Thus, summarized, we have that
    \begin{align*}
      \overline{\supC_{1+k}\|\supC_{2+k}\|L_C} 
        & = \overline{\supC_{1+k}\|L_C \parallel \supC_{2+k}\|L_C}\\
        & = \overline{\supC_{1+k}\|L_C} \parallel \overline{\supC_{2+k}\|L_C}\\
        & = \overline{\supC_{1+k}} \| \overline{\supC_{2+k}} \| \overline{L_C}\,.
    \end{align*}

    To prove controllability, note that 
      $\supC_{i+k}$ is controllable with respect to $\overline{\supC_{i+k}}$, for $i=1,2$, and
      $L_C$ is controllable with respect to $\overline{P_0(\supC_{1+k})} \| \overline{P_0(\supC_{2+k})}$.
    Now we use Lemma~\ref{feng} several times, and the nonconflictness shown above, to obtain that 
    \begin{itemize}
      \item $\supC_{i+k} \parallel L_C$ is controllable with respect to $(\overline{\supC_{i+k}}) \parallel (\overline{P_0(\supC_{1+k})} \| \overline{P_0(\supC_{2+k})})$, for $i=1,2$,
    
      \item $(\supC_{1+k}\|L_C) \parallel (\supC_{2+k}\|L_C) = \supC_{1+k}\|\supC_{2+k}\|L_C$ is controllable with respect to $(\overline{\supC_{1+k}} \| \overline{P_0(\supC_{1+k})} \| \overline{P_0(\supC_{2+k})}) \parallel (\overline{\supC_{2+k}} \| \overline{P_0(\supC_{1+k})} \| \overline{P_0(\supC_{2+k})})$ that can be simplified to $\overline{\supC_{1+k}} \| \overline{\supC_{2+k}}$,
    
      \item $\overline{\supC_{1+k}} \parallel \overline{\supC_{2+k}}$ is controllable with respect to $(L(G_1)\|\overline{\supC_k}) \parallel (L(G_2)\|\overline{\supC_k})=L(G_1)\|L(G_2)\|\overline{\supC_k}$, and
      \item $L(G_1)\|L(G_2)\|\overline{\supC_k}$ is controllable with respect to $L(G_1)\|L(G_2)\|L(G_k)$ because the language $\supC_k$ is controllable with respect to $L(G_k)$.
    \end{itemize}
    Using transitivity of controllability, Lemma~\ref{lem_trans}, we obtain that $\supC_{1+k}\|\supC_{2+k}\|L_C$ is controllable with respect to $L(G_1)\|L(G_2)\|L(G_k)=L(G_1)\|L(G_2)$, because the coordinator $G_k$ is constructed in such a way that it does not change the plant.
  \qed\end{proof}
  
  To demonstrate this improvement, we consider Example~\ref{ex3}.
  \begin{example}
    Consider the generators of Example~\ref{ex3}. Note that $G_1\|G_2\|G_0$, Fig.~\ref{fig7}, is not controllable with respect to the plant $G_1\|G_2$, Fig.~\ref{fig5}, if $b$ is uncontrollable. The generator $G_0=P(G_1)\|P(G_2)$ is depicted in Fig.~\ref{fig10}. 
    \begin{figure}[htb]
    \centering
    \begin{tikzpicture}[->,>=stealth,shorten >=1pt,auto,node distance=1.4cm,
      state/.style={circle,minimum size=6mm,very thin,draw=black,initial text=}]
      \node[state,initial,accepting] (1) {1};
      \node[state,accepting] (2) [right of=1] {2};
      \node[state] (3) [right of=2] {3};
      \path
      (1) edge node {$a$} (2)
      (2) edge node {$b$} (3);
    \end{tikzpicture}
    \caption{Generator $P(G_1)\|P(G_2)$}\label{fig10}
    \end{figure}
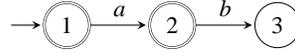
    It is not hard to see that if $b$ is not controllable, then the supremal controllable sublanguage of $L_m(G_0)$ with respect to $L(G_0)$ is $L_C=\{\eps\}$, because event $a$ must be prevent from happening. Therefore, the language of $L(G_1\|G_2)\|L_C=\{\eps\}$ as expected.
  \end{example}

  We can now summarize this method as an algorithm.
  \begin{algorithm}[Coordinator for nonblockingness]
    Consider the notation above.
    \begin{enumerate}
      \item Compute $\supC_{1+k}$ and $\supC_{2+k}$ as defined in (\ref{eq0}).
      \item Let $\Sigma_0:=\Sigma_k$ and $P_0:=P_k$.
      \item Extend the event set $\Sigma_0$ so that the projection $P_0$ is both a $\supC_{1+k}$- and a $\supC_{2+k}$-observer.
      \item Define the coordinator $C$ as the minimal nonblocking generator such that $L_m(C)=\supC(P_0(\supC_{1+k})\parallel P_0(\supC_{2+k}),\ \overline{P_0(\supC_{1+k})}\parallel \overline{P_0(\supC_{2+k})},\ \Sigma_{0,u})$.
    \end{enumerate}
  \end{algorithm}

  This algorithm (Step~1) is based on the computation of the languages $\supC_{1+k}$ and $\supC_{2+k}$ defined in (\ref{eq0}), which can be computed using a standard algorithm for the computation of supremal controllable sublanguages. If the assumption of Theorem~\ref{thm2} is satisfied, the computed languages are the languages of local supervisors that are the candidates to solve the problem. However, the composition $\supC_{1+k} \| \supC_{2+k}$ can be blocking, and a coordinator for nonblockingness is required.
  
  In Step~2, we define a new event set $\Sigma_0$ (and the corresponding projection) that is initialized to be the event set $\Sigma_k$ used in the computation in Step~1. 
  
  In Step~3 of the algorithm, the event set $\Sigma_0$ must be extended so that the projection $P_0$ is both a $\supC_{1+k}$- and $\supC_{2+k}$-observer. Thus, in consequence of the extension operation, $\Sigma_k$ can become a proper subset of $\Sigma_0$. Even though the computation of such a minimal extension is NP-hard, a polynomial algorithm computing a reasonable extension exists, cf. \cite{LFobs} for more details and the algorithm.
  
  Finally, in Step~4, the coordinator generator $C$ is defined as the minimal nonblocking generator accepting the supremal controllable sublanguage of the language $P_0(\supC_{1+k})\parallel P_0(\supC_{2+k})$ with respect to the language $\overline{P_0(\supC_{1+k})} \parallel \overline{P_0(\supC_{2+k})}$. This idea has been used by Feng in~\cite{FLT}. In other words, if $S_1$ and $S_2$ are generators for languages $\supC_{1+k}$ and $\supC_{2+k}$, respectively, then the coordinator $C$ is computed as the generator for the supremal controllable sublanguage of $P_0(S_1)\|P_0(S_2)$. Since $P_0$ is an observer, the computation can be done in polynomial time, cf.~\cite{Won04}.

  \begin{remark}
    In the previous section we discussed the case when $\supC_k\not\subseteq P_k(\supC_{i+k})$, for $i=1,2$. Note that the coordinator $L_C$ discussed in this section can also be used in that case because $\supC_{1+k}\| L_C$ and $\supC_{2+k}\|L_C$ then form synchronously nonconflicting local supervisors such that their overall behavior is controllable with respect to the global plant. Hence, although this solution may not be optimal, it presents a solution in the case of (non-prefix-closed) languages that do not satisfy the assumptions of Theorem~\ref{thm2}, or of those of Section~\ref{sec:revisited} in the case of prefix-closed languages.
  \end{remark}
  
\section{Supremal prefix-closed languages}\label{sec:revisited}
  In this section, we revise the case of prefix-closed languages. We use the local control consistency property (LCC) instead of the output control consistency property (OCC), cf.~\cite{automatica2011}. The reason for this is that LCC is a less restrictive condition than OCC, as shown in~\cite[Lemma 4.4]{SB11}. Moreover, the extension of our approach to an arbitrary number of local plants is sketched.
  
  \begin{theorem}\label{thm2pc}
    Let $K$ be a prefix-closed sublanguage of the plant language $L$, where $L=L(G_1\| G_2\| G_k)$, and $G_i$ is a generator over $\Sigma_i$, for $i=1,2,k$. Assume that the language $K$ is conditionally decomposable, and define the languages $\supC_k$, $\supC_{1+k}$, $\supC_{2+k}$ as in (\ref{eq0}). Let the projection $P^{i+k}_k$ be an $(P^{i+k}_i)^{-1}(L(G_i))$-observer and LCC for the language $(P^{i+k}_i)^{-1}(L(G_i))$, for $i=1,2$. Then 
    \[
      \supC_{1+k} \parallel \supC_{2+k} = \supCC(K, L, (\Sigma_{1,u}, \Sigma_{2,u}, \Sigma_{k,u}))\,.
    \]
  \end{theorem}
  \begin{proof}
    In this proof, let $\supCC$ denote the supremal conditionally controllable language $\supCC(K, L, (\Sigma_{1,u}, \Sigma_{2,u}, \Sigma_{k,u}))$, and $M$ the parallel composition $\supC_{1+k} \parallel \supC_{2+k}$. It is shown in~\cite[Theorem~11]{automatica2011} that $\supCC$ is a subset of $M$ and that $M$ is a subset of $K$. To prove that $P_k(M) \Sigma_{k,u} \cap L(G_k)$ is a subset of $P_k(M)$, consider a word $x$ in $P_k(M)$ and an uncontrollable event $a$ in $\Sigma_{k,u}$ such that the word $xa$ is in $L(G_k)$. To show that the word $xa$ is in $P_k(M)= P^{1+k}_k(\supC_{1+k}) \cap P^{2+k}_k(\supC_{2+k})$, note that there exists a word $w$ in $M$ such that $P_k(w)=x$. It is shown in~\cite[Theorem~11]{automatica2011} that there exists a word $u$ in $(\Sigma_1\setminus \Sigma_k)^*$ such that the word $P_{1+k}(w)ua$ is in $(P^{1+k}_1)^{-1}(L(G_1))$ and the word $P_{1+k}(w)$ is in $L(G_1) \parallel \supC_k$. As the projection $P^{1+k}_k$ is LCC for the language $(P^{1+k}_1)^{-1}(L(G_1))$, there exists a word $u'$ in $(\Sigma_u\setminus \Sigma_k)^*$ such that $P_{1+k}(w)u'a$ is in $(P^{1+k}_1)^{-1}(L(G_1))$. Then, controllability of $\supC_{1+k}$ implies that $P_{1+k}(w)u'a$ is in $\supC_{1+k}$, that is, $xa$ is in $P^{1+k}_k(\supC_{1+k})$. Analogously, we can prove that $xa$ is in $P^{2+k}_k(\supC_{2+k})$. Thus, $xa$ is in $P_k(M)$. The rest of the proof is the same as in~\cite[Theorem~11]{automatica2011}.
  \qed\end{proof}

  In this Theorem, a relatively large number of properties that the coordinator, the local plants and the specification have to satisfy is assumed. However, a polynomial algorithm extending the coordinator event set so that the language $K$ becomes conditionally decomposable has already been discussed, see~\cite{SCL12}. In addition, to ensure that the projection $P^{i+k}_k$ is an $(P^{i+k}_i)^{-1}(L(G_i))$-observer and LCC for the language $(P^{i+k}_i)^{-1}(L(G_i))$, the coordinator event set can again be extended so that the conditions are fulfilled~\cite{SB11,LFobs}.
  
  Conditions of Theorem~\ref{thm2pc} imply that the projection $P_k$ is LCC for the language $L$.
  \begin{lemma}\label{lem46}
    Let $G_i$ over $\Sigma_i$ be generators, for $i=1,2$. Let $\Sigma=\Sigma_1\cup \Sigma_2$, and let $P_i: \Sigma^*\to \Sigma_i^*$, for $i=1,2,k$ and $\Sigma_k\subseteq \Sigma$, be projections. If $\Sigma_1\cap \Sigma_2$ is a subset of $\Sigma_k$ and the projection $P^{i+k}_k$ is LCC for the language $(P^{i+k}_i)^{-1}(L(G_i))$, for $i=1,2$, then the projection $P_k$ is LCC for the language $L=L(G_1\| G_2\| G_k)$.
  \end{lemma}
  \begin{proof}
    For a word $s$ in $L$ and an event $\sigma_u$ in $\Sigma_{k,u}$, assume that there exists a word $u$ in $(\Sigma\setminus \Sigma_k)^*$ such that $su\sigma_u$ is in $L$. Then $P_{i+k}(su\sigma_u) = P_{i+k}(s)P_{i+k}(u)\sigma_u$ is in $(P^{i+k}_i)^{-1}(L(G_i))$ implies that there exists a word $v_{i}$ in $(\Sigma_{i+k,u}\setminus \Sigma_k)^*$, for $i=1,2$, such that $P_{i+k}(s)v_{i}\sigma_u$ is in $(P^{i+k}_i)^{-1}(L(G_i))$. As $P_k(v_{i})=\eps$, $P_i(v_{i})=v_i$ and we get that $P_i(s)P_i(v_i)P_i(\sigma_u)$ is in $L(G_i)$, for $i=1,2,k$. Consider a word $u'$ in $\{v_1\}\| \{v_2\}$. Then $P_i(u')=v_i$ and, thus, $su'\sigma_u$ is in $L$. Moreover, $u'$ is in $(\Sigma_u\setminus \Sigma_k)^*$.
  \qed\end{proof}

  It is an open problem how to verify that the projection $P_{i+k}$ is LCC for the language $L$ without computing the whole plant. In such a case and with the coordinator language included in the corresponding projection of the plant language, the solution computed using our coordination control architecture coincides with the global optimal solution given by the supremal controllable sublanguage of the specification.
  \begin{theorem}\label{thm4}
    Consider the setting of Theorem~\ref{thm2pc}. If, in addition, $L(G_k)$ is a subset of $P_k(L)$ and the projection $P_{i+k}$ is LCC for the language $L$, for $i=1,2$, then 
    \[
      \supC(K, L, \Sigma_{u}) = \supCC(K, L, (\Sigma_{1,u}, \Sigma_{2,u}, \Sigma_{k,u}))\,.
    \]
  \end{theorem}
  \begin{proof}
    It was shown in~\cite[Theorem~15]{automatica2011} that the projection $P_k$ is an $L$-observer. Moreover, by Lemma~\ref{lem46}, the projection $P_k$ is LCC for the language $L$. Let $\supC$ denote $\supC(K,L,\Sigma_u)$. We prove that the language $P_k(\supC)$ is controllable with respect to $L(G_k)$. Consider a word $t$ in $P_k(\supC)$ and an event $a$ in $\Sigma_{k,u}$ such that the word $ta$ is in $L(G_k)$, which is a subset of $P_k(L)$. We proved in~\cite[Theorem~15]{automatica2011} that there exist words $s$ in $\supC$ and $u$ in $(\Sigma\setminus \Sigma_k)^*$ such that $sua$ is in $L$ and $P_k(sua)=ta$. By the LCC property of the projection $P_k$, there exists a word $u'$ in $(\Sigma_{u}\setminus \Sigma_k)^*$ such that $su'a$ is in $L$. By controllability of the language $\supC$ with respect to $L$, the word $su'a$ is in $\supC$, that is, $P_k(su'a)=ta$ is in $P_k(\supC)$. Thus, (1) of Definition~\ref{def:conditionalcontrollability} holds. By~\cite[Theorem~15]{automatica2011}, the projection $P_{i+k}$ is an $L$-observer, for $i=1,2$. To prove (2) of Definition~\ref{def:conditionalcontrollability}, consider a word $t$ in $P_{i+k}(\supC)$, for $1\le i\le 2$, and an event $a$ in $\Sigma_{i+k,u}$ such that the word $ta$ is in $L(G_i)\parallel P_k(\supC)$. We proved in~\cite[Theorem~15]{automatica2011} that there exist words $s$ in $\supC$ and $u$ in $(\Sigma\setminus \Sigma_k)^*$ such that $sua$ is in $L$ and $P_{i+k}(sua)=ta$. As the projection $P_{i+k}$ is LCC for the language $L$, there exists a word $u'$ in $(\Sigma_{u}\setminus \Sigma_{1+k})^*$ such that $su'a$ is in $L$. Then controllability of $\supC$ with respect to $L$ implies that $su'a$ is in $\supC$, that is, $P_{i+k}(su'a)=ta$ is in $P_{i+k}(\supC)$. The other inclusion is the same as in~\cite[Theorem~15]{automatica2011}.
  \qed\end{proof}
  
   Finally, a natural and simple extension to more than two local subsystems with one central coordinator is sketched. All concepts and results carry over to this general case of $n$ subsystems, where the coordinator event set $\Sigma_k$ should contain all shared events (events common to two or more subsystems). Conditional decomposability is then simply decomposability with respect to event sets $(\Sigma_{i})_{i=1}^{n}$ and $\Sigma_{k}$, cf. Section~\ref{sec:nphard}. It is a very good news for large systems that conditional decomposability can be checked in polynomial time with respect to the number of components as has been noticed in~\cite{SCL12}. Note that unlike the previous form of conditional controllability,  Definition~\ref{def:conditionalcontrollability} can be extended to the general case of $n$ subsystems in an obvious way. Namely, conditions (2) and (3) are replaced by $n$ conditions of the form $P_{i+k}(K)$ is controllable with respect to $L(G_i) \parallel \overline{P_k(K)}$ and $\Sigma_{i+k,u}$. 
   
   Note, however, that for many large-scale systems a single central coordinator might be of little (if any) help due to too many events to be included in the coordinator event sets so that the conditions presented in this paper are satisfied (in particular, conditional decomposability, LCC, and observer conditions). It is always possible to relax some of the assumptions with the price of losing optimality, but in future publications we will rather propose multi-level coordination architectures with several layers of coordinators together with  different optimality conditions corresponding to a given multi-level coordination architecture.

\section{Conclusion}\label{sec:conclusion}
  We have revised, simplified, and extended the coordination control scheme for dis\-crete-event systems. These results have been used, for the case of prefix-closed languages, in the implementation of the coordination control plug-in for libFAUDES. We have identified cases, where supremal conditionally-controll\-able sublanguages can be computed even in the case of non-prefix-closed specification languages, and proposed coordinators for nonblockingness in addition to coordinators for safety developed in our earlier publications. Note that a general procedure for the computation of supremal conditionally-controll\-able sublanguages in the case of non-prefix-closed specification languages is still missing.
  
  Another aspect that requires further investigation is the generalization of coordination control from the current case of one central coordinator to multilevel  coordination control with several coordinators on different levels. In fact, one central coordinator is typically not enough in the case of large number of local subsystems, because too many events must be communicated (added into the coordinator event set) between the coordinator and local subsystems. This general architecture will be computationally more efficient, because less events need to be communicated. In the multi-level  coordination control the subsystems will be organized into different groups and each group will have a coordinator meaning that only events from a given group will be communicated among subsystems of the same group via the coordinator. 

\appendix
\section{Auxiliary results}\label{appendix}
  In this section, we list auxiliary results required in the paper.
  \begin{lemma}[Proposition~4.6, \cite{FLT}]\label{feng}
    Let $L_i$ over $\Sigma_i$, for $i=1,2$, be prefix-closed languages, and let $K_i$ be a controllable sublanguage of $L_i$ with respect to $L_i$ and $\Sigma_{i,u}$. Let $\Sigma=\Sigma_1\cup \Sigma_2$. If $K_1$ and $K_2$ are synchronously nonconflicting, then $K_1\parallel K_2$ is controllable with respect to $L_1\parallel L_2$ and $\Sigma_u$.
  \end{lemma}

  \begin{lemma}[\cite{automatica2011}]\label{lem_trans}
    Let $K$ be a subset of a language $L$, and $L$ be a subset of a language $M$ over $\Sigma$ such that $K$ is controllable with respect to $\overline{L}$ and $\Sigma_u$, and $L$ is controllable with respect to $\overline{M}$ and $\Sigma_u$. Then $K$ is controllable with respect to $\overline{M}$ and $\Sigma_u$.
  \end{lemma}

  \begin{lemma}[\cite{Won04}]\label{lemma:Wonham}
    Let $P_k : \Sigma^*\to \Sigma_k^*$ be a projection, and let $L_i$ be a language over $\Sigma_i$, where $\Sigma_i$ is a subset of $\Sigma$, for $i=1,2$, and $\Sigma_1\cap \Sigma_2$ is a subset of $\Sigma_k$. Then $P_k(L_1\| L_2)=P_k(L_1) \| P_k(L_2)$.
  \end{lemma}

  \begin{lemma}[\cite{automatica2011}]\label{lem11}
    Let $L_i$ be a language over $\Sigma_i$, for $i=1,2$, and let $P_i : (\Sigma_1\cup \Sigma_2)^* \to \Sigma_i^*$ be a projection. Let $A$ be a language over $\Sigma_1\cup \Sigma_2$ such that $P_1(A)$ is a subset of $L_1$ and $P_2(A)$ is a subset of $L_2$. Then $A$ is a subset of $L_1\parallel L_2$.
  \end{lemma}

  \begin{lemma}[\cite{pcl06}]\label{fengT41}
    Let $L_i$ be a language over $\Sigma_i$, for $i\in J$, and let $\cup_{k,\ell\in J}^{k\neq\ell} (\Sigma_k\cap \Sigma_\ell)\subseteq \Sigma_0$. If $P_{i,0}:\Sigma_i^* \to (\Sigma_i\cap \Sigma_0)^*$ is an $L_i$-observer, for $i\in J$, then $\overline{\|_{i\in J} L_i} = \|_{i\in J} \overline{L_i}$ if and only if $\overline{\|_{i\in J} P_{i,0}(L_i)} = \|_{i\in J} \overline{P_{i,0}(L_i)}$.
  \end{lemma}
  
  \begin{lemma}[\cite{scl2011}]\label{CDlemma}
    A language $K\subseteq (\Sigma_1\cup \Sigma_2\cup\ldots\cup \Sigma_n)^*$ is conditionally decomposable with respect to event sets $\Sigma_1$, $\Sigma_2$,\ldots, $\Sigma_n$, $\Sigma_k$ if and only if there exist languages $M_{i+k}\subseteq \Sigma_{i+k}^*$, $i=1,2,\ldots,n$, such that $K=\parallel_{i=1}^{n} M_{i+k}$.
  \end{lemma}

\subsubsection*{Acknowledgements.}
  Research of two first authors was supported by RVO:~67985840. In addition, the first author was supported by the Grant Agency of the Czech Republic under grant P103/11/0517 and the second author by the Grant Agency of the Czech Republic under grant P202/11/P028.
  
\bibliographystyle{spmpsci}
\bibliography{wodes12ex}

\end{document}